\documentclass[preprint,11pt]{article}
\usepackage{fullpage}
\usepackage{amssymb,amsfonts,amsmath,amsthm,amscd,dsfont,mathrsfs}
\usepackage{graphicx,float,psfrag,epsfig,amssymb}
\usepackage[usenames,dvipsnames,svgnames,table]{xcolor}
\definecolor{darkgreen}{rgb}{0.0,0,0.9}
\usepackage[pagebackref,letterpaper=true,colorlinks=true,pdfpagemode=none,citecolor=OliveGreen,linkcolor=BrickRed,urlcolor=BrickRed,pdfstartview=FitH]{hyperref}
\usepackage{wrapfig}
\usepackage{relsize}
\usepackage{color}
\usepackage{pict2e}
\usepackage[tight]{subfigure}
\usepackage{algorithm}
\usepackage[noend]{algorithmic}
\usepackage{caption}
\usepackage{nameref}

\DeclareMathAlphabet{\mathpzc}{OT1}{pzc}{m}{it}

\newtheorem{propo}{Proposition}[section]
\newtheorem{lemma}[propo]{Lemma}

\newtheorem{coro}[propo]{Corollary}
\newtheorem{thm}[propo]{Theorem}

\newtheorem{remark}[propo]{Remark}
\newtheorem{claim}[propo]{Claim}

\def\tX{\widetilde{X}}
\def\hS{\widehat{S}}
\def\ed{\stackrel{\rm d}{=}}

\def\tX{\widetilde{X}}

\def\cE{{\cal E}}

\def\reals{{\mathbb R}}

\def\prob{{\mathbb P}}
\def\E{{\mathbb E}}
\def\Var{{\rm Var}}

\def\L0{{L_0}}

\def\<{\langle}
\def\>{\rangle}

\def\bZ{{\mathbf Z}}
\def\bX{{\mathbf X}}

\def\hr{\widehat{r}}
\def\htheta{\widehat{\theta}}
\def\hSigma{\widehat{\Sigma}}

\def\supp{{\rm supp}}

\def\F{{\sf F}}

\def\F{{\sf F}}
\def\normal{{\sf N}}

\def\tX{{\tilde{X}}}

\def\sT{{\sf T}}

\def\id{{\rm I}}

\def\sign{{\rm sign}}
\def\hu{\widehat{u}}

\def\smax{\sigma_{\max}}

\def\smin{\sigma_{\min}}

\def\event{\mathcal{E}}
\def\term{\mathcal{T}}
\def\v*{v_0}
\def\T*{T_0}
\def\xistar{\xi_0}
\def\u*{u_0}
\def\F*{F_0}
\def\thetaGL{\htheta^{{\rm GL}}}
\def\thetaZN{\htheta^{{\rm ZN}}}
\definecolor{olivegreen}{rgb}{0,0.6,0.4}
\def\hkappa{\widehat{\kappa}}

\newcommand{\ajcomment}[1]{}
\newcommand{\deq}{\stackrel{\text{\rm d}}{=}}
\makeatletter
\newcommand{\labitem}[2]{%
\def\@itemlabel{\text{#1}}
\item
\def\@currentlabel{#1}\label{#2}}
\makeatother

\addtocontents{toc}{\protect\setcounter{tocdepth}{2}}

\title{Model Selection for High-Dimensional Regression under the Generalized Irrepresentability Condition}

\author{Adel Javanmard
            \footnote{Department of Electrical Engineering, Stanford University. Email: \url{adelj@stanford.edu} }
             \,and Andrea~Montanari 
            \footnote{Department of Electrical Engineering and Department of Statistics, Stanford University. Email: \url{montanar@stanford.edu}}
            }

\begin{document}

\maketitle

\begin{abstract}
In the high-dimensional regression model a response
variable is linearly related to $p$ covariates, but the sample size
$n$ is smaller than $p$. We assume that only a small subset of
covariates is `active' (i.e., the corresponding coefficients are
non-zero), and consider the model-selection problem of identifying the
active covariates.

A popular approach is to estimate the regression coefficients through
the Lasso ($\ell_1$-regularized least squares). This is known to
correctly identify the active set only if the irrelevant covariates are roughly orthogonal to
the relevant ones, as quantified through the so called
`irrepresentability' condition. In this paper we study the
`Gauss-Lasso' selector, a simple two-stage method that first solves
the Lasso, and then performs ordinary least squares restricted to the
Lasso active set. 

We formulate `generalized irrepresentability condition' (GIC), an assumption that is substantially weaker than 
irrepresentability. We prove that,
under GIC, the Gauss-Lasso correctly recovers the active set. 
\end{abstract}

\begingroup
\hypersetup{linktocpage}
\tableofcontents
\endgroup
%==============================================================
\section{Introduction}

In linear regression, we wish to estimate an unknown but fixed vector of parameters $\theta_0 \in \reals^p$
from $n$ pairs $(Y_1,X_1), (Y_2,X_2), \dotsc, (Y_n,X_n)$, with vectors
$X_i$ taking values in $\reals^p$ and response variables $Y_i$ given by
\begin{eqnarray}\label{eqn:regression}
Y_i \,=\, \<\theta_0,X_i\> + W_i\, ,\;\;\;\;\;\;\;\; W_i\sim
\normal(0,\sigma^2)\, ,
\end{eqnarray}
where $\<\,\cdot\,,\,\cdot\,\>$ is the standard scalar product. 

In matrix form,
letting  $Y = (Y_1,\dotsc,Y_n)^\sT$ and denoting by $\bX$ the design matrix with
rows $X_1^\sT,\dotsc, X_n^\sT$, we have
\begin{eqnarray}\label{eq:NoisyModel}
Y\, =\, \bX\,\theta_0+ W\, ,\;\;\;\;\;\;\;\; W\sim
\normal(0,\sigma^2 \id_{n\times n})\, .
\end{eqnarray}
In this paper, we consider the high-dimensional setting in which the number of
parameters exceeds the sample size, i.e., $p > n$, but the number of
non-zero entries of $\theta_0$ is smaller than $p$.
We denote by $S \equiv \supp(\theta_0)\subseteq[p]$ the support of
$\theta_0$, and let $s_0\equiv |S|$. We are interested in the `model
selection' problem, namely in the problem of identifying $S$ from data
$Y$, $\bX$. 

In words, there exists a `true' low dimensional linear model that explains the
data. We want to identify the set $S$ of  covariates that are
`active' within this model.
This problem has motivated a large body of research, because
of its relevance to several modern data analysis tasks, ranging from
signal processing \cite{Donoho1,CandesFourier} to genomics \cite{peng2010regularized,shevade2003simple}. 
A crucial step forward has been the
development of model-selection techniques based on convex optimization
formulations \cite{Tibs96,BP95,Dantzig}. These formulations have lead to computationally efficient algorithms  that
can be applied to large scale problems. Such developments pose the following theoretical question: 
\emph{For which vectors $\theta_0$, designs $\bX$, and noise levels $\sigma$,
the support $S$ can be identified, with high probability, through
computationally efficient procedures?} The same question
can be asked for random designs $\bX$ and, in this case, `high probability'
will refer both to the noise realization $W$, and to the design
realization $\bX$. In the rest of this introduction we shall focus
--for the sake of simplicity-- on the deterministic settings, and
refer to Section \ref{sec:RandomDesigns} for a treatment of Gaussian
random designs.

The analysis of computationally efficient methods has largely focused on
$\ell_1$-regularized least squares, a.k.a.  the Lasso \cite{Tibs96}.
The Lasso estimator is defined by 
\begin{align}
\htheta^{n}(Y,\bX;\lambda) \equiv \arg\min_{\theta\in\reals^p}
\Big\{\frac{1}{2n}\|Y-\bX\theta\|^2_2+\lambda\|\theta\|_1\Big\}\, . \label{eq:LassoEstimator}
\end{align}
In case the right hand side has more than one minimizer, one of them
can be selected arbitrarily for our purposes.
We will often omit the arguments $Y$, $\bX$, as they are clear from
the context. (A closely related method is the so-called Dantzig
selector \cite{Dantzig}: it would be interesting to explore whether our
results can be generalized to that approach.)

It was understood early on that, even in the large-sample, low-dimensional limit
$n\to\infty$ at $p$ constant, $\supp(\htheta^n)\neq S$ unless the
columns of $\bX$ with index in $S$ are roughly orthogonal to the ones
with index outside $S$ \cite{knight2000asymptotics}. This assumption is formalized by the so-called
`irrepresentability condition', that can be stated in
terms of the empirical covariance matrix $\hSigma =  (\bX^\sT \bX/n)$.
Letting $\hSigma_{A,B}$ be the submatrix $(\hSigma_{i,j})_{i\in
  A,j\in B}$, irrepresentability requires
\begin{align}
\|\hSigma_{S^c,S}\hSigma_{S,S}^{-1}\,\sign(\theta_{0,S})\|_{\infty} \le 1-\eta\, ,\label{eq:IRR}
\end{align}
for some $\eta> 0$
 (here $\sign(u)_i =+1$,
$0$, $-1$ if, respectively, $u_i>0$, $=0$, $<0$). In an early breakthrough, Zhao and Yu \cite{zhao}
proved that, if this condition holds with $\eta$ uniformly bounded
away from $0$, it  guarantees correct model selection
also in the high-dimensional regime $p\gg n$. Meinshausen and
B\"ulmann \cite{MeinshausenBuhlmann} independently established the
same result for random Gaussian designs, with applications to learning
Gaussian graphical models. These papers applied to very sparse
models, requiring in particular $s_0 = O(n^c)$, $c<1$, and
parameter vectors with large coefficients. Namely, scaling the columns of $X$
such that $\hSigma_{i,i}\le 1$, for $i\in [p]$, they require
$\theta_{\rm min}\equiv \min_{i\in S}|\theta_{0,i}|\ge c\sqrt{s_0/n}$.

Wainwright \cite{Wainwright2009LASSO} strengthened considerably these
results by allowing for general scalings of $s_0,p,n$ and proving that
much smaller non-zero coefficients can be detected. Namely, he proved that for a broad class of empirical covariances
it is
only necessary that $\theta_{\rm min}\ge c\sigma\sqrt{(\log p)/n}$.
This scaling of the minimum non-zero entry is optimal up to
constants. Also, for a specific classes of random Gaussian designs
(including $\bX$ with i.i.d. standard Gaussian entries), the analysis
of \cite{Wainwright2009LASSO} provides tight bounds on the minimum
sample size for correct model selection. Namely, there exists
$c_{\ell}, c_u>0$ such that the Lasso fails with high probability if
$n<c_\ell \, s_0\log p$ and succeeds with high probability if $n\ge
c_{u}\, s_0\log p$. 

While, thanks to these recent works \cite{zhao,MeinshausenBuhlmann,Wainwright2009LASSO}, 
we understand reasonably well  model selection via the
Lasso, it is fundamentally unknown what model-selection performances
can be achieved with general computationally practical methods. Two
aspects of of the above theory cannot be improved substantially: $(i)$ The non-zero
entries must satisfy the condition $\theta_{\rm min}\ge
c\sigma/\sqrt{n}$ to be detected with high probability.
Even if $n=p$ and the measurement directions $X_i$ are orthogonal,
e.g., $\bX = \sqrt{n} \id_{n\times n}$, one would need $|\theta_{0,i}| \ge c\sigma/\sqrt{n}$ to distinguish
the $i$-th entry from noise. For instance, in~\cite{javanmard2013hypothesis}, 
the present authors prove a general upper bound on the minimax power
of tests for hypotheses $H_{0,i}= \{\theta_{0,i} = 0\}$.
Specializing this bound to the case of standard Gaussian designs, the
analysis of \cite{javanmard2013hypothesis}
shows formally that no test can detect $\theta_{0,i} \neq 0$, with a fixed degree of confidence, unless 
$|\theta_{0,i}| \ge c \sigma/\sqrt{n}$. 
$(ii)$ The sample size must satisfy $n\ge s_0$. Indeed, if this is
not the case, for each $\theta_0$ with support of size
$|S|=s_0$, there is a one parameter family $\{\theta_0(t)=\theta_0+t\,
v\}_{t\in\reals}$ with $\supp(\theta_0(t))\subseteq S$, $\bX\theta_0(t) = \bX\theta_0$ and, for
specific values of $t$, the support of $\theta_0(t)$ is strictly
contained in $S$.

On the other hand, there is no fundamental reason to assume the
irrepresentability condition (\ref{eq:IRR}). This follows  from the
requirement that a specific method (the Lasso) succeeds, but is
unclear why it should be necessary in general.  
The situation is very different for estimation consistency, e.g., for
characterizing the $\ell_2$ error $\|\htheta-\theta_0\|_2$. In that case
the restricted isometry property (RIP) \cite{CandesTao} (or one of its relaxations
\cite{BickelEtAl,BuhlmannVanDeGeer}) is sufficient and --essentially-- necessary.

In this paper we prove that the \emph{Gauss-Lasso selector} has
nearly optimal model selection properties under a condition that is
strictly weaker than irrepresentability. We call this condition the
\emph{generalized irrepresentability condition} (GIC). 
The Gauss-Lasso procedure uses the Lasso estimator to estimate a first
model $T\subseteq\{1,\dots,p\}$. It then constructs
a new estimator by ordinary least squares regression of 
the data $Y$ onto the model $T$.

\begin{algorithm}[t]
\caption*{{\sc Gauss-Lasso selector}:  Model selector for high dimensional problems}
\begin{algorithmic}[1]

\REQUIRE Measurement vector $y$, design model $\bX$, regularization parameter $\lambda$, support size $s_0$. 

\ENSURE Estimated support $\hS$.

\STATE Let $T =\supp(\htheta^{n})$ be the support of Lasso estimator $\htheta^{n} = \htheta^{n}(y,\bX,\lambda)$
 given by
 \[
 \htheta^{n}(Y,\bX;\lambda) \equiv \arg\min_{\theta\in\reals^p}
\Big\{\frac{1}{2n}\|Y-\bX\theta\|^2_2+\lambda\|\theta\|_1\Big\}\, . 
\]

\STATE Construct the estimator $\thetaGL$ as follows:
\[
\thetaGL_{T} = (\bX_T^\sT \bX_T)^{-1} \bX_T^\sT y\,,  \quad \quad \thetaGL_{T^c} = 0\,.
\]

\STATE Find $s_0$-{th} largest entry (in modulus) of $\thetaGL_{T}$,
denoted by $\thetaGL_{(s_0)}$, and let
\[
\hS \equiv \big\{i\in [p]:\; |\thetaGL_i|\ge |\thetaGL_{(s_0)}| \big\}.
\]
\end{algorithmic}
\end{algorithm}

We prove that the estimated model is, with high probability, correct
(i.e., $\hS=S$) under conditions comparable to  the ones assumed in
\cite{MeinshausenBuhlmann,zhao,Wainwright2009LASSO},
while replacing irrepresentability by the weaker generalized
irrepresentability condition.
In the case of random Gaussian designs, our analysis further assumes
the restricted eigenvalue property in order to establish a nearly
optimal scaling of the sample size $n$ with the sparsity parameter $s_0$.

In order to build some intuition about the difference between
irrepresentability and generalized irrepresentability, it is
convenient to consider the Lasso cost function at `zero noise':
\begin{align*}
G(\theta;\xi) &\equiv
\frac{1}{2n}\|\bX(\theta-\theta_0)\|_2^2+\xi\|\theta\|_1\\
&=\frac{1}{2}\<(\theta-\theta_0),\hSigma(\theta-\theta_0)\>
+\xi\|\theta\|_1\,. 
\end{align*}
Let $\thetaZN(\xi)$ be the minimizer of $G(\,\cdot\,;\xi)$ and 
$v \equiv \lim_{\xi\to 0+}\sign(\thetaZN(\xi))$. The limit is well
defined by Lemma \ref{lem:LimitXi0-DET} below.
The KKT conditions for $\thetaZN$ imply, for $T\equiv\supp(v)$, 
\begin{align*}
\|\hSigma_{T^c,T}\hSigma_{T,T}^{-1}v_T\|_{\infty} \le 1\, .
\end{align*}
Since $G(\,\cdot\,;\xi)$ has always at least one minimizer, this condition
is \emph{always satisfied}.
Generalized irrepresentability requires that the above inequality
holds with some small slack $\eta>0$ bounded away from zero, i.e.,
\begin{align*}
\|\hSigma_{T^c,T}\hSigma_{T,T}^{-1}v_T\|_{\infty} \le 1-\eta\, .
\end{align*}
Notice that this assumption reduces to standard irrepresentability
cf. Eq.~(\ref{eq:IRR})  if, in addition, we ask that $v =
\sign(\theta_0)$. 
In other words, earlier work
  \cite{MeinshausenBuhlmann,zhao,Wainwright2009LASSO} required generalized irrepresentability
\emph{plus} sign-consistency in zero noise, and established sign
consistency in non-zero noise. In this paper the former condition is shown to
be sufficient.

From a different point of view, GIC demands that 
irrepresentability holds for a superset of the true support $S$. It
was indeed argued in the literature that such a relaxation of
irrepresentability allows to cover a significantly broader set of
cases (see for instance \cite[Section 7.7.6]{buhlmann2011statistics}). 
However, it was never clarified why such a superset irrepresentability
condition should be significantly more general than simple
irrepresentability. Further, no precise prescription existed for the
superset of the true support.

Our contributions can therefore be summarized as follows:
\begin{enumerate}
\item  By tying it to the KKT condition for the zero-noise problem, 
we justify the expectation that generalized irrepresentability should
hold for a broad class of design matrices.
\item We thus provide a specific formulation of superset
  irrepresentability, prescribing both the superset $T$ and the sign
  vector $v_T$,
that is --by itself-- significantly more general than simple irrepresentability.
\item We show that, under GIC, exact support recovery can be
  guaranteed using the Gauss-Lasso, and formulate the appropriate
  `minimum coefficient' conditions that guarantee this.
\end{enumerate}
As a side remark, even when simple irrepresentability holds,
our results strengthen somewhat the estimates of
\cite{Wainwright2009LASSO} (see below for details).

\vspace{0.2cm}

The paper is organized as follows. In the rest of the introduction we
illustrate the range of applicability of GIC through a simple example and
we discuss further related work. We finally introduce the basic
notations to be used throughout the paper.

Section \ref{sec:Deterministic} treats the case of deterministic
designs $\bX$, and develops our main results on the basis of the
GIC. Section \ref{sec:RandomDesigns} extends our analysis to the case
of random designs. In this case GIC is required to hold for the
population covariance, and the analysis is more technical as it
requires to control the randomness of the design matrix.
The proofs of our main results can be found in Sections
\ref{proof:two-thm} and \ref{proof:LassoSuppRecovery},
with several technical steps deferred to the Appendices.

\subsection{An example}
\label{subsec:example}

In order to illustrate the range of new cases covered by our results, 
it is instructive to consider a simple
example. A detailed discussion of this calculation can be found in
Appendix \ref{app:Example}.
The example corresponds to a Gaussian random design, i.e., the rows 
$X_1^{\sT}$, \dots $X_n^{\sT}$ are i.i.d. realizations of a $p$-variate normal
distribution with mean zero. We write
$X_i=(X_{i,1},X_{i,2},\dots,X_{i,p})^{\sT}$ for the components of $X_i$.
The response variable is linearly related to the first $s_0$
covariates
\begin{align*}
Y_i =
\theta_{0,1}X_{i,1}+\theta_{0,2}X_{i,2}+\cdots+\theta_{0,s_0}X_{i,s_0}
+ W_i\, ,
\end{align*}
where $W_i\sim \normal(0,\sigma^2)$ and we assume $\theta_{0,i}>0$ for all $i\le
s_0$.  In particular $S=\{1,\dots,s_0\}$.

As for the design matrix, first $p-1$ covariates are orthogonal at the
population level,
i.e., $X_{i,j}\sim\normal(0,1)$ are independent for $1\le j\le p-1$
(and $1\le i \le n$). However the $p$-th covariate is correlated to
the $s_0$ relevant ones:
\begin{align*}
X_{i,p} = a\, X_{i,1}+a\,X_{i,2}+\dots+a\, X_{i,s_0} +
b\, \tX_{i,p}\, .
\end{align*}
Here $\tX_{i,p}\sim\normal(0,1)$ is independent from
$\{X_{i,1},\dots,X_{i,p-1}\}$ and represents the orthogonal component
of the $p$-th covariate. We choose the  coefficients $a,b\ge 0$ such that
$s_0a^2+b^2=1$, whence $\E\{X_{i,p}^2\}=1$ and hence the $p$-th
covariate is normalized as the first $(p-1)$ ones.
In other words,  the rows of $\bX$ are i.i.d. Gaussian
$X_i\sim\normal(0,\Sigma)$ with covariance given by
\begin{align*}
\Sigma_{ij} = \begin{cases}
1 & \mbox{if $i=j$},\\
a & \mbox{if $i=p, j\in S$ or $i\in S, j=p$,}\\
0 & \mbox{otherwise.}
\end{cases}
\end{align*}

For $a=0$, this is the standard i.i.d. design and irrepresentability
holds. The Lasso correctly recovers the support $S$ from $n\ge c\,
s_0\log p$ samples, provided $\theta_{\rm min}\ge c'\sqrt{(\log
  p)/n}$. It follows from \cite{Wainwright2009LASSO} that this remains
true as long as $a\le (1-\eta)/s_0$ for some $\eta>0$ bounded away
from $0$. However, as soon as $a>1/s_0$, the Lasso includes the
$p$-th covariate in the estimated model, with high probability (see Appendix
\ref{app:Example}). 

As it is shown in Appendix \ref{app:Example}, the Gauss-Lasso is successful
for a significantly larger set of values of $a$. Namely, if
\begin{align*}  
a\in
\left[0,\frac{1-\eta}{s_0}\right]\cup
\left(\frac{1}{s_0},\frac{1-\eta}{\sqrt{s_0}}\right]\, ,
\end{align*}
then it recovers $S$ from $n\ge c\,
s_0\log p$ samples, provided $\theta_{\rm min}\ge c'\sqrt{(\log
  p)/n}$. While the interval $((1-\eta)/s_0,1/s_0]$ is not
covered by this result, we expect this to be due to the proof
technique rather than to an intrinsic limitation of the Gauss-Lasso selector.
 
%%%%%%%%%%%%%%%%%%%%%%%%
\subsection{Further related work}

The restricted isometry property \cite{CandesTao,Dantzig} (or the related
restricted eigenvalue \cite{BickelEtAl} or compatibility conditions~\cite{BuhlmannVanDeGeer})
have been used to establish guarantees on the estimation and model selection errors of the Lasso or similar
approaches. In particular, Bickel, Ritov and Tsybakov
\cite{BickelEtAl} show that, under such conditions, with high probability,
\begin{align*}
\|\htheta-\theta_0\|^2_2\le C \sigma^2 \frac{s_0 \log p}{n}\, .
\end{align*}

The same conditions can be used to prove model-selection guarantees.
In particular, Zhou \cite{Zhou-threshold} studies a multi-step
thresholding procedure whose first steps coincide with the
Gauss-Lasso. While the main objective of this work is to prove
high-dimensional $\ell_2$ consistency with a sparse estimated model, the author
also proves partial model selection guarantees. Namely, the method
correctly recovers a subset of large coefficients $S_L\subseteq S$,
provided  $|\theta_{0,i}|\ge c\sigma\sqrt{s_0(\log p)/n}$, for $i\in
S_L$.  This means that the coefficients that are guaranteed to be
detected must be a factor $\sqrt{s_0}$ larger than what is required by
our results. 

Also related to model selection is the recent line of work on
hypothesis testing in high-dimensional regression
\cite{ZhangZhangSignificance,BuhlmannSignificance}. These papers
propose methods for testing hypotheses of the form $H_{0,i}= \{\theta_{0,i}=0\}$.
In order to achieve a given significance level, they require --again--
large coefficients, namely $|\theta_{0,i}|\ge c\sigma\sqrt{s_0(\log
  p)/n}$ (see \cite{javanmard2013hypothesis} for a discussion of this point). 
In \cite{javanmard2013hypothesis}, we investigate a
hypothesis testing method that achieves any given significance level $\alpha$ for
$|\theta_{0,i}|\ge c\sigma/\sqrt{n}$, with $c$ a constant that depends on $\alpha$. 
Although the testing procedure can be used for general setting, the guarantee
on its statistical power is provided only for some random Gaussian designs in an asymptotic sense. 
A very recent paper by van de Geer, B{\"u}hlmann and Ritov
\cite{GBR-hypothesis} proposes a similar procedure and gives conditions
 under which the procedure achieves the semiparametric efficiency bound. 
Their analysis allows for general Gaussian and sub-Gaussian designs.
However, it requires a sample size
$n\ge C (s_0\log p)^2$, namely the square of the optimal sample size.

Let us finally mention that an alternative approach to establishing
model-selection guarantees assumes a suitable mutual incoherence conditions. Lounici
\cite{lounici2008sup} proves correct model selection under the
assumption $\max_{i\neq j}|\hSigma_{ij}| = O(1/s_0)$.  This assumption
is however stronger than irrepresentability \cite{BuhlmannVanDeGeer}.
Cand\'es and Plan \cite{candes2009near} also assume mutual
incoherence, albeit with a much weaker requirement, namely $\max_{i\neq
  j}|\hSigma_{ij}| = O(1/(\log p))$. Under this condition, they
establish model selection guarantees for an ideal scaling of the
non-zero coefficients $\theta_{\rm min}\ge c\sigma\sqrt{(\log
  p)/n}$. However, this result only holds with high probability for a
`random signal model' in which the non-zero coefficients $\theta_{0,i}$
have uniformly random signs.

Finally, model selection consistency can be obtained without
irrepresentability through other methods. For instance
\cite{zou2006adaptive} develops the adaptive Lasso,  using
a data-dependent weighted $\ell_1$ regularization, and
\cite{bach2008bolasso} proposes the Bolasso, a resampling-based
techniques. Unfortunately, both of these approaches are only
guaranteed to succeed in the low-dimensional regime of $p$ fixed, and $n\to\infty$.

%%%%%%%%%%%%%%%%%%%%%%
\subsection{Notations}
We provide a brief summary of the notations used throughout the paper.
For a matrix $A$ and set of indices $I,J$, we let $A_{J}$ denote the submatrix
containing just the columns in $J$ and $A_{I,J}$ denote the submatrix formed 
by the rows in $I$ and columns in $J$. Likewise, for a vector $v$, $v_I$ is the restriction
of $v$ to indices in $I$. Further, the notation $A^{-1}_{I,I}$ represents the inverse of $A_{I,I}$, i.e., $A^{-1}_{I,I} = (A_{I,I})^{-1}$.
The maximum and the minimum singular values of $A$ are respectively denoted 
by $\sigma_{\max}(A)$ and $\sigma_{\min}(A)$.
We write $\|v\|_p$ for the standard $\ell_p$ norm of a vector $v$.
Specifically, $\|v\|_0$ denotes the number of nonzero entries in $v$. Also, $\|A\|_p$ refers
to the induced operator norm on a matrix $A$. We use $e_i$ to refer to the $i$-th standard basis element, e.g., 
$e_1 = (1,0,\dotsc,0)$.
For a vector $v$, $\supp(v)$ represents
the positions of nonzero entries of $v$. Throughout, we denote the rows of the design matrix $\bX$ by $X_1,\dotsc,X_n \in \reals^p$ and
denote its columns by $x_1,\dotsc,x_p \in \reals^n$. 
Further, for a vector $v$,
$\sign(v)$ is the vector with 
entries $\sign(v)_i = +1$ if $v_{i}>0$, $\sign(v)_i=-1$ if
$v_{i}<0$, and $\sign(v)_i=0$ otherwise. 
%=========================================================
\section{Deterministic designs}
\label{sec:Deterministic}

An outline of this section is given below:
\begin{enumerate}
\item We first consider the zero-noise problem $W = 0$,
and prove several useful properties
of the Lasso estimator in this case. 
In particular, we show that there exists a threshold for the regularization parameter
below which the support of the Lasso estimator remains the same and contains $\supp(\theta_0)$.
Moreover, the Lasso estimator support is not much larger than $\supp(\theta_0)$.

\item We then turn to the noisy problem, and introduce the \emph{generalized
irrepresentability condition} (GIC)  that is motivated by the properties of the Lasso in the zero-noise case. We prove that under GIC  (and other technical conditions), with high probability, the signed support of the Lasso estimator is the same as that in the  zero-noise problem.

\item We show that the Gauss-Lasso selector correctly
recovers the signed support of $\theta_0$.

\end{enumerate}
%

% 
%===========================
 \subsection{Zero-noise problem}
 \label{sec:zero-noise}
Recall that $\hSigma \equiv (\bX^\sT \bX/n)$ denotes the empirical
covariance of the rows of the design matrix. 
Given $\hSigma \in \reals^{p\times p}$, $\hSigma \succeq 0$, $\theta_0 \in \reals^p$ and $\xi \in \reals_+$, we define the \emph{zero-noise Lasso estimator}
as 
\begin{align}
\thetaZN(\xi) \equiv \arg\min_{\theta\in\reals^p}
\Big\{\frac{1}{2n}\<(\theta-\theta_0), \hSigma(\theta-\theta_0)\> + \xi\|\theta\|_1\Big\}\, . \label{eq:ZNLassoEstimator}
\end{align}
Note that $\thetaZN(\xi)$ is obtained by letting $Y = \bX\theta_0$ in the definition of $\htheta^n(Y,\bX;\xi)$.

Following \cite{BickelEtAl}, we introduce a restricted
eigenvalue constant for the empirical covariance matrix $\hSigma$:
\begin{align}\label{eq:RE-DET}
\hkappa(s,c_0) \equiv \min_{\substack{J\subseteq [p]\\ |J|\le s}} 
\min_{\substack{u\in\reals^p\\ \|u_{J^c}\|_1\le c_0\|u_J\|_1}}\,
\frac{\<u,\hSigma u\>}{\|u\|^2_2}\, .
\end{align}

Our first result states that the support of $\thetaZN(\xi)$ is not
much larger than the support of $\theta_0$, for any $\xi>0$.
\begin{lemma}\label{lem:Tsize-DET}
Let $\thetaZN =\thetaZN(\xi)$ be defined as per
Eq.~(\ref{eq:NinftyProblem}), with $\xi>0$. Then, if $s_0 = \|\theta_0\|_0$,
\begin{align}\label{eqn:TsizeB-DET}
\|\thetaZN\|_0\le \left(1+\frac{4\|\hSigma\|_2}{\hkappa(s_0,1)}\right)\, s_0\, .
\end{align} 
\end{lemma}
The proof of this lemma is deferred to Section~\ref{proof:Tsize-DET}.

\begin{lemma}\label{lem:LimitXi0-DET}
Let $\thetaZN =\thetaZN(\xi)$ be defined as per
Eq.~(\ref{eq:ZNLassoEstimator}), with $\xi>0$. Then there exist $\xi_0 = \xi_0(\hSigma, S, \theta_0) > 0$,
$\T* \subseteq [p]$, $v_0 \in \{-1,0,+1\}^p$, such that 
the following happens. For all $\xi\in (0,\xistar)$,
$\sign(\thetaZN(\xi)) = \v*$ and $\supp(\thetaZN(\xi))
=\supp(\v*) = \T*$.
Further $\T*\supseteq S$, $v_{0,S} = \sign(\theta_{0,S})$ and
$\xi_0 = \min_{i\in S} |\theta_{0,i} / [\hSigma_{\T*,\T*}^{-1}v_{0,\T*}]_i |$.
\end{lemma}
Proof of Lemma~\ref{lem:LimitXi0-DET} can be found in Section~\ref{proof:LimitXi0-DET}.

Finally we have the following standard characterization of the
solution of the zero-noise problem.
\begin{lemma}\label{lem:ZN-supp}
Let $\thetaZN =\thetaZN(\xi)$ be defined as per
Eq.~(\ref{eq:ZNLassoEstimator}), with $\xi>0$.
Let $T\supseteq S$ and $v\in \{+1,0,-1\}^p$ be such that $\supp(v) =
T$. Then  $\sign(\thetaZN)=v$ if and only if
\begin{align}
&\Big\|\hSigma_{T^c,T}\hSigma_{T,T}^{-1}v_T\Big\|_{\infty}\le
1\, ,\label{eq:thetaZNvT1}\\
v_T &= \sign
\Big(\theta_{0,T}-\xi\hSigma_{T,T}^{-1}v_T\Big)\,.\label{eq:thetaZNvT2}
\end{align}
Further, if the above holds,
$\thetaZN$ is given by $\thetaZN_{T^c}=0$ and
\begin{align*}
\thetaZN_T & =
\theta_{0,T}-\xi\hSigma_{T,T}^{-1}v_T\, .
\end{align*}
\end{lemma}
Lemma~\ref{lem:ZN-supp} is proved in Appendix~\ref{app:ZN-supp}.

Motivated by this result, we introduce the \emph{generalized irrepresentability condition} (GIC) for deterministic designs.
\begin{itemize}
\item[] {\bf Generalized irrepresentability (deterministic designs).} The pair
$(\hSigma,\theta_0)$, $\hSigma\in\reals^{p\times p}$, $\theta_0\in
\reals^p$ satisfy the generalized irrepresentability condition with
parameter $\eta>0$ if the following happens.  Let 
$\v* $, $\T*$ be defined as per Lemma \ref{lem:LimitXi0-DET}. Then 
\begin{align}
\Big\|\hSigma_{\T*^c,\T*}\hSigma_{\T*,\T*}^{-1} v_{0,\T*}\Big\|_{\infty}
&\le 1-\eta\, .\label{eq:GIC-DET}
\end{align}
\end{itemize}
In other words we require the dual feasibility condition
(\ref{eq:thetaZNvT1}) --which always holds-- to hold with a positive
slack $\eta$.

%=========================================================
\subsection{Noisy problem}
Consider the noisy linear observation model as described in~\eqref{eq:NoisyModel}, and let $\hr \equiv (\bX^\sT W/n)$.
We begin with a standard characterization of $\sign(\htheta^{n})$, the signed support of the Lasso estimator~\eqref{eq:LassoEstimator}.

\begin{lemma}\label{lem:supp-lasso-DET}
Let $\htheta^n = \htheta^n(y,\bX;\lambda)$ be defined as per
Eq.~(\ref{eq:LassoEstimator}), and let $z\in\{+1,0,-1\}^p$ with $\supp(z) = T$. Further assume $T\supseteq
S$. Then the signed support of the Lasso estimator is given by $\sign(\htheta^n)=z$ if and only if 
\begin{align}
\Big\|
\hSigma_{T^c,T}&\hSigma_{T,T}^{-1}z_T + \frac{1}{\lambda} (\hr_{T^c} - \hSigma_{T^c,T}\hSigma_{T,T}^{-1}\hr_{T})
\Big\|_{\infty}\le
1\, , \label{eq:z1-DET}\\
&z_T = \sign\Big(\theta_{0,T} - \hSigma^{-1}_{T,T} (\lambda z_T - \hr_T) \Big)\,.\label{eq:z2-DET}
\end{align}
\end{lemma}
Lemma~\ref{lem:supp-lasso-DET} is proved in Appendix~\ref{app:supp-lasso-DET}.

\begin{thm}\label{thm:LassoSuppRecovery-DET}
Consider the deterministic design model with empirical covariance matrix $\hSigma \equiv (\bX^\sT\bX)/n$,
and assume that $\hSigma_{i,i} \le 1$ for $i\in[p]$.
Let $\T*\subseteq [p]$, $\v*\in\{+1,0,-1\}^p$ be the 
 set and vector defined in Lemma~\ref{lem:LimitXi0-DET}, and $t_0 \equiv |\T*|$.
 Assume that 
\begin{itemize}
\labitem{(i)}{Condition:Minsigma-DET} We have $\sigma_{\min}(\hSigma_{\T*,\T*}) \ge C_{\min} > 0$. 
\item[(ii)] The pair $(\hSigma,\theta_0)$ satisfies the generalized
  irrepresentability condition with parameter $\eta$.
\end{itemize}
Consider the Lasso estimator $\htheta^n = \htheta^n(y,\bX;\lambda)$ defined as per
Eq.~(\ref{eq:LassoEstimator}), with regularization parameter
\begin{align}
\lambda = \frac{\sigma}{\eta} \sqrt{\frac{2 c_1 \log p}{n}}\,,\label{eq:lambda_val-DET}
\end{align}
for some constant $c_1 > 1$, and suppose that
\begin{enumerate}
\labitem{(iii)}{Condition:ThetaMin-DET} For some $c_2>0$:
\begin{align}
|\theta_{0,i}|\ge  c_2\lambda+ \lambda\big|[\hSigma_{\T*,\T*}^{-1}v_{0,\T*}]_i\big|\;\;\;\;\;\;\; &\mbox{ for
  all }i\in S,\label{eq:ConditionTheta1-DET}\\
\big|[\hSigma_{\T*,\T*}^{-1}v_{0,\T*}]_i \big|\ge c_2 \;\;\;\;\;\;\; &\mbox{ for
  all }i\in \T*\setminus S. \label{eq:ConditionTheta2-DET}
\end{align}
\end{enumerate}
We further assume, without loss of generality, $\eta\le
c_2\sqrt{C_{\min}}$. Then the following holds true:
\begin{align}
\prob\Big\{\sign(\htheta^n(\lambda)) = \v*\Big\}\ge 1- 4 p^{1-c_1}\,.
\label{MainThm1:ProbabilityEstimate}
\end{align}
\end{thm}
Theorem~\ref{thm:LassoSuppRecovery-DET} is proved in Section~\ref{proof:LassoSuppRecovery-DET}.
Note that, even in the case standard irrepresentability holds
(and hence $\T*=S$), this result improves over \cite[Theorem
1.(b)]{Wainwright2009LASSO}, in that the required lower bound for $|\theta_{0,i}|$, $i \in S$,
does not depend on $\|\hSigma_{S,S}\|_\infty$. More precisely, Theorem~\ref{thm:LassoSuppRecovery-DET} assumes $|\theta_{0,i}| \ge \lambda (c_2 + |[\hSigma^{-1}_{S,S} v_{0,S}]_i|)$,
for $i \in S$, which is weaker than the assumption of Theorem1.(b)\cite{Wainwright2009LASSO}, namely,
$|\theta_{0,i}| \ge \lambda (c + \|\hSigma^{-1}_{S,S} \|_{\infty})$, since $\|v_{0,S}\|_{\infty} \le 1$. 

\begin{remark}
Condition~\ref{Condition:Minsigma-DET} in Theorem~\ref{thm:LassoSuppRecovery-DET} requires the submatrix $\hSigma_{\T*,\T*}$ to have minimum singular value bounded away form zero.
 %This is a reasonable assumption since $\T*$ is not much larger than
 %$S$, as stated in Lemma~\ref{lem:Tsize-DET}.
Assuming $\hSigma_{S,S}$ to be non-singular is necessary for
identifiability. Requiring the minimum singular value of
$\hSigma_{\T*,\T*}$ to be bounded away from zero is not much more
restrictive since $\T*$ is comparable in size with $S$, as stated in Lemma~\ref{lem:Tsize-DET}.
\end{remark}

We next show that the Gauss-Lasso selector correctly recovers the support of $\theta_0$. 
\begin{thm}\label{thm:GLSuppRecovery-DET}
Consider the deterministic design model with empirical covariance matrix $\hSigma \equiv (\bX^\sT\bX)/n$,
and assume that $\hSigma_{i,i} \le 1$ for $i\in[p]$. Under the assumptions of Theorem~\ref{thm:LassoSuppRecovery-DET},
\[
\prob\Big(\|\thetaGL- \theta_0\|_\infty \ge \mu \Big)
\le 4p^{1-c_1} + 2p e^{-nC_{\min}\mu^2/ 2\sigma^2}\,.
\]
In particular, if $\hS$ is the model selected by the Gauss-Lasso, we have
\[
\prob(\hS = S) \ge 1 - 6\, p^{1-c_1/4} \,.
\]
\end{thm}

The proof of Theorem~\ref{thm:GLSuppRecovery-DET} is given in Section~\ref{proof:GLSuppRecovery-DET}.
%=========================================================
\section{Random Gaussian designs}
\label{sec:RandomDesigns}

In the previous section, we studied the case of deterministic design models which allowed for
a straightforward analysis. Here, we consider the random design model which needs a
more involved analysis.   
Within the random Gaussian design model, the rows $X_i$ are distributed as
$X_i\sim \normal(0,\Sigma)$ for some (unknown) covariance matrix $\Sigma\succ
0$. 

In order to study the performance of Gauss-Lasso selector in this case, we first define the population-level estimator.
Given $\Sigma\in\reals^{p\times p}$, $\Sigma\succ 0$,
$\theta_0\in\reals^p$ and $\xi \in
\reals_+$, the \emph{population-level estimator} $\htheta^{\infty}(\xi) =
\htheta^{\infty}(\xi;\theta_0,\Sigma)$ is defined
as
\begin{align}
\htheta^{\infty}(\xi)\equiv \arg\min_{\theta\in\reals^p}\Big\{
\frac{1}{2}\,\<(\theta-\theta_0),\Sigma(\theta-\theta_0)\> + \xi
\|\theta\|_1\Big\}\, .\label{eq:NinftyProblem}
\end{align}
Notice that the minimizer is unique because $\Sigma$ is strictly
positive definite and hence the cost function on the right-hand side
is strongly convex. In fact, the population-level estimator is obtained by assuming that
the response vector $Y$ is noiseless and $n= \infty$, hence replacing the empirical covariance
$(\bX^\sT \bX /n)$ with the exact covariance $\Sigma$ in the lasso optimization problem~\eqref{eq:LassoEstimator}.

Notice that the population-level estimator $\htheta^{\infty}$ is deterministic, albeit $\bX$ is a random design.
We show that under some conditions on the covariance $\Sigma$ and vector $\theta_0$,
$T\equiv \supp(\htheta^{n}) = \supp(\htheta^{\infty})$, i.e., the population-level estimator and the Lasso estimator
share the same (signed) support. 
Further $T \supseteq S$. Since $\htheta^{\infty}$ (and hence $T$) is deterministic, $\bX_T$ is a 
Gaussian matrix with rows drawn independently from $\normal(0,\Sigma_{T,T})$. This observation
allows for a simple analysis of the Gauss-Lasso selector $\thetaGL$.  
 
An outline of the section is given below:
\begin{enumerate}
\item We begin with proving several properties
of the population-level estimator.  
Similar to the zero-noise problem in Section~\ref{sec:zero-noise}, we show that there exists a threshold $\xi_0$, 
such that for all $\xi \in (0,\xi_0)$, $\supp(\htheta^{\infty}(\xi))$ remains the same and contains $\supp(\theta_0)$.
Moreover, $\supp(\htheta^\infty(\xi))$ is not much larger than $\supp(\theta_0)$.

\item We show that under GIC for covariance matrix $\Sigma$ (and other sufficient conditions), with high probability, the signed support of the Lasso estimator is the same as the signed support of the population-level estimator.

\item Following the previous steps, we show that the Gauss-Lasso selector correctly recovers  the signed support of $\theta_0$.

\end{enumerate}
%

%================================================
\subsection{The $n=\infty$ problem}
In this section we derive several useful properties of the
population-level problem (\ref{eq:NinftyProblem}).
Comparing Eqs.~\eqref{eq:ZNLassoEstimator} and~\eqref{eq:NinftyProblem},
the estimators $\thetaZN(\xi)$ and $\htheta^\infty(\xi)$ are defined in a very similar manner
(the former is defined with respect to $\hSigma$ and the latter is defined with respect to $\Sigma$), and
as we will see $\htheta^\infty$ also possesses the properties stated in Section~\ref{sec:zero-noise}.

Let $\kappa_\infty(s,c_0)$ be the restricted eigenvalue constant for the covariance matrix $\Sigma$:
\begin{align}\label{eq:RE}
\kappa(s,c_0) \equiv \min_{\substack{J\subseteq [p]\\ |J|\le s}} 
\min_{\substack{u\in\reals^p\\ \|u_{J^c}\|_1\le c_0\|u_J\|_1}}\,
\frac{\<u,\Sigma u\>}{\|u\|^2_2}\, .
\end{align}
%
 
%We also define
%%
%\[
%\theta_{\min} \equiv \min_{i\in S} |\theta_{0,i}|\,.
%\]
%% 

The proofs of the following Lemmas are very similar to the corresponding ones in Section~\ref{sec:zero-noise},
 and are omitted.

\begin{lemma}\label{lem:Tsize}
Let $\htheta^{\infty} =\htheta^{\infty}(\xi)$ be defined as per
Eq.~(\ref{eq:NinftyProblem}), with $\xi>0$. Then, if $s_0 = \|\theta_0\|_0$,
\begin{align}\label{eqn:TsizeB}
\|\htheta^{\infty}\|_0\le \left(1+\frac{4\|\Sigma\|_2}{\kappa(s_0,1)}\right)\, s_0\, .
\end{align} 
\end{lemma}
\begin{lemma}\label{lem:LimitXi0}
Let $\htheta^{\infty} =\htheta^{\infty}(\xi)$ be defined as per
Eq.~(\ref{eq:NinftyProblem}), with $\xi>0$. Then there exist $\xi_0 = \xi_0(\Sigma, S, \theta_0) > 0$,
$\T* \subseteq [p]$, $v_0 \in \{-1,0,+1\}^p$, such that 
the following happens. For all $\xi\in (0,\xistar)$,
$\sign(\htheta^{\infty}(\xi)) = \v*$ and $\supp(\htheta^{\infty}(\xi))
=\supp(\v*) = \T*$.
Further $\T*\supseteq S$, $v_{0,S} = \sign(\theta_{0,S})$ and
$\xi_0 = \min_{i\in S} |\theta_{0,i} / [\Sigma_{\T*,\T*}^{-1}v_{0,\T*}]_i |$.
\end{lemma}

Finally we have the following standard characterization of the
solution of the $n=\infty$ problem (\ref{eq:NinftyProblem}).
\begin{lemma}\label{lem:infinity-supp}
Let $\htheta^{\infty} =\htheta^{\infty}(\xi)$ be defined as per
Eq.~(\ref{eq:NinftyProblem}), with $\xi>0$.
Let $T\supseteq S$ and $v\in \{+1,0,-1\}^p$ be such that $\supp(v) =
T$. Then  $\sign(\htheta^{\infty})=v$ if and only if
\begin{align*}
&\Big\|\Sigma_{T^c,T}\Sigma_{T,T}^{-1}v_T\Big\|_{\infty}\le
1\, ,\\
v_T &= \sign
\Big(\theta_{0,T}-\xi\Sigma_{T,T}^{-1}v_T\Big)\, .
\end{align*}
Further, if the above holds,
$\htheta^{\infty}$ is given by $\htheta^{\infty}_{T^c}=0$ and
\begin{align*}
\htheta^{\infty}_T & =
\theta_{0,T}-\xi\Sigma_{T,T}^{-1}v_T\, .
\end{align*}
\end{lemma}
Motivated by this result, we introduce the following assumption.
\begin{itemize}
\item[] {\bf Generalized irrepresentability (random designs).} The pair
$(\Sigma,\theta_0)$, $\Sigma\in\reals^{p\times p}$, $\theta_0\in
\reals^p$ satisfy the generalized irrepresentability condition with
parameter $\eta>0$ if the following happens.  Let 
$\v* $, $\T*$ be defined as per Lemma \ref{lem:LimitXi0}. Then 
\begin{align}
\Big\|\Sigma_{\T*^c,\T*}\Sigma_{\T*,\T*}^{-1} v_{0,\T*}\Big\|_{\infty}
&\le 1-\eta\, ,\label{eq:GIC-RAN}
\end{align}
\end{itemize}
%

%=====================================================
\subsection{The high-dimensional problem}

We now consider the Lasso estimator (\ref{eq:LassoEstimator}). 
Recall the notations 
\begin{align*}
\hSigma\equiv\frac{1}{n}\bX^{\sT}\bX\,,\quad \quad
\hr  \equiv \frac{1}{n}\bX^{\sT}W\, .
\end{align*}
%
%We will drop the superscript $n$ whenever clear from the context.
Note that $\hSigma \in\reals^{p\times p}$,
$\hr\in\reals^p$ are both random quantities in the case of random designs.

\begin{thm}\label{thm:LassoSuppRecovery}
Consider the Gaussian random design model with covariance matrix $\Sigma \succ 0$,
and assume that $\Sigma_{i,i} \le 1$ for $i\in[p]$.
Let $\T*\subseteq [p]$, $\v*\in\{+1,0,-1\}^p$ be the deterministic
 set and vector defined in Lemma~\ref{lem:LimitXi0}, and $t_0 \equiv |\T*|$.
 Assume that 
\begin{itemize}
\labitem{(i)}{Condition:Minsigma} We have 
% $\sigma_{\max}(\Sigma_{\T*,\T*}) \le C_{\max}$, and
 $\sigma_{\min}(\Sigma_{\T*,\T*}) \ge C_{\min} > 0$. 
\item[(ii)] The pair $(\Sigma,\theta_0)$ satisfies the generalized
  irrepresentability condition with parameter $\eta$.
\end{itemize}
Consider the Lasso estimator $\htheta^n = \htheta^n(y,\bX;\lambda)$ defined as per
Eq.~(\ref{eq:LassoEstimator}), with regularization parameter
\begin{align}
\lambda = \frac{4\sigma}{\eta} \sqrt{ \frac{c_1 \log p}{n}}\,,\label{eq:lambda_val}
\end{align}
for some constant $c_1 > 1$, and suppose that
\begin{enumerate}
\labitem{(iii)}{Condition:ThetaMin} For some $c_2>0$:
\begin{align}
|\theta_{0,i}|\ge  c_2\lambda+\frac{3}{2}\lambda\big|[\Sigma_{\T*,\T*}^{-1}v_{0,\T*}]_i\big|\;\;\;\;\;\;\; &\mbox{ for
  all }i\in S,\label{eq:ConditionTheta1}\\
\big|[\Sigma_{\T*,\T*}^{-1}v_{0,\T*}]_i \big|\ge 2c_2 \;\;\;\;\;\;\; &\mbox{ for
  all }i\in \T*\setminus S. \label{eq:ConditionTheta2}
\end{align}
\end{enumerate}
We further assume, without loss of generality, $\eta\le
c_2\sqrt{C_{\min}}$.

If $n \ge \max(M_1,M_3) t_0 \log p$ with 
\begin{eqnarray*}
M_1  \equiv  \frac{74c_1}{\eta^2C_{\min}}\, ,\quad \quad 
M_3 \equiv \frac{32^2 c_1}{c_2^2 C_{\min}^2}  \, ,
\end{eqnarray*}
 then the following holds true:
\begin{align}
\prob\Big\{\sign(\htheta^n(\lambda)) = \v*\Big\}\ge 1 -  p e^{-\frac{n}{10}} - 6e^{-\frac{t_0}{2}} - 8 p^{1-c_1}\,.
\label{MainThm:ProbabilityEstimate}
\end{align}
\end{thm}
Under standard irrepresentability, this result improves over \cite[Theorem
3.(ii)]{Wainwright2009LASSO}, in that the required lower bound for $|\theta_{0,i}|$, $i \in S$,
does not depend on $\|\Sigma_{S,S}^{-1/2}\|_{\infty}$.
More precisely, Theorem~\ref{thm:LassoSuppRecovery-DET} assumes $|\theta_{0,i}| \ge \lambda (c_2 + 1.5 |[\Sigma^{-1}_{S,S} v_{0,S}]_i|)$,
for $i \in S$, while Theorem 3.(ii)\cite{Wainwright2009LASSO} requires
$|\theta_{0,i}| \ge  c \lambda \|\Sigma^{-1/2}_{S,S} \|^2_{\infty}$, for $i \in S$. Note that
 $|[\Sigma^{-1}_{S,S} v_{0,S}]_i| \le \|\Sigma^{-1}_{S,S}\|_\infty \le \|\Sigma^{-1/2}_{S,S}\|^2_\infty$.
 %and in general $\|\Sigma^{-1/2}_{S,S}\|^2_\infty= O(s_0)$, whereas $\|\Sigma^{-1}_{S,S}\|_\infty = O(\sqrt{s_0})$.

While being closely analogous to Theorem
\ref{thm:LassoSuppRecovery-DET}, the last theorem  has somewhat worse
constants. Indeed in the present case we need to control the
randomness of the design matrix $\bX$ in addition to the one of the
noise.
\begin{remark}
Condition~\ref{Condition:Minsigma-DET} follows readily from the restricted eigenvalue constraint as in Eq.~\eqref{eq:RE}, i.e., $\kappa_\infty(t_0,0) > 0$. This is a reasonable assumption since $\T*$ is not much larger than $S_0$, as stated in Lemma~\ref{lem:Tsize}.
\end{remark}

\begin{coro}\label{cor:t0->s0}
Under the assumptions of Theorem~\ref{thm:LassoSuppRecovery}, if $n \ge \max(\widetilde{M}_1, \widetilde{M}_3) s_0 \log p$, with
\[
\widetilde{M}_1 =  \Big(1+ \frac{4\|\Sigma\|_2}{\kappa_\infty(s_0,1)} \Big) M_1 \,, \quad \quad
\widetilde{M}_3 =  \Big(1+ \frac{4\|\Sigma\|_2}{\kappa_\infty(s_0,1)} \Big) M_3 \,,
\] 
then the following holds:
\[
\prob\Big\{\sign(\htheta^n(\lambda)) = \v*\Big\}\ge 1 -  
pe^{-\frac{n}{10}} - 6e^{-\frac{s_0}{2}} - 8 p^{1-c_1}\,.
\]
\end{coro}
\begin{proof}[Proof (Corollary~\ref{cor:t0->s0})]
The result follows readily from Theorem~\ref{thm:LassoSuppRecovery}, noting that $s_0 \le t_0$ since $S_0 \subseteq \T*$,
and $t_0 \le (1 + 4\|\Sigma\|_2/\kappa_\infty(s_0,1)) s_0$ as per Lemma~\ref{lem:Tsize}.
\end{proof}

Below, we show that the Gauss-Lasso selector correctly recovers the signed support of $\theta_0$.
\begin{thm}\label{thm:GLSuppRecovery}
Consider the random Gaussian design model with covariance matrix $\Sigma \succ 0$,
and assume that $\Sigma_{i,i} \le 1$ for $i\in[p]$. Under the assumptions of Theorem~\ref{thm:LassoSuppRecovery}, and
for $n \ge \max(\widetilde{M}_1,\widetilde{M}_3) s_0 \log p$, we have
\[
\prob\Big(\|\thetaGL- \theta_0\|_\infty \ge \mu \Big) \le 
pe^{-\frac{n}{10}} + 6e^{-\frac{s_0}{2}} + 8 p^{1-c_1} + 2p e^{-nC_{\min}\mu^2/ 2\sigma^2}\,.
\]
Moreover, letting $\hat{S}$ be the model returned by the Gauss-Lasso selector,  we have
\[
\prob(\hS = S) \ge 1 - p\,e^{-\frac{n}{10}} - 6\,e^{-\frac{s_0}{2}} - 10\, p^{1-c_1} \,.
\]
\end{thm}
The proof of Theorem~\ref{thm:GLSuppRecovery} is deferred to Section~\ref{proof:GLSuppRecovery}.
\begin{remark}{\bf [Detection level]} 
Let $\theta_{\min} \equiv \min_{i\in S} |\theta_{0,i}|$ be the minimum magnitude of the non-zero entries of vector
$\theta_0$. 
By Theorem~\ref{thm:GLSuppRecovery}, Gauss-Lasso selector correctly recovers $\supp(\theta_0)$, with
probability greater than $1 - p\,e^{-\frac{n}{10}} - 6\,e^{-\frac{s_0}{2}} - 10\, p^{1-c_1}$, if 
$n \ge \max(\tilde{M}_1,\tilde{M}_3) s_0 \log p$, and
\begin{eqnarray}\label{eq:thetamin}
\theta_{\min} \ge C \sigma  \sqrt{\frac{\log p}{n}}\, \big(1 + \|\Sigma^{-1}_{\T*,\T*}\|_\infty\big)\,,
\end{eqnarray}
where $C = C(c_1,c_2,\eta)$ is a constant depending on $c_1,c_2$, and $\eta$. Eq.~\eqref{eq:thetamin}
stems from the condition~\ref{Condition:ThetaMin} in Theorem~\ref{thm:LassoSuppRecovery}.

We can further generalize this result. Define
\[
S_1 = \bigg\{i\in S: |\theta_{0,i}| \ge C \sigma  \sqrt{\frac{\log p}{n}}\, \big(1 + \|\Sigma^{-1}_{\T*,\T*}\|_\infty\big) \bigg\}\,, 
\]
and $S_2 = S\backslash S_1$. 
By a very similar argument to the proof of Theorem~\ref{thm:LassoSuppRecovery},
the Gauss-Lasso selector can recover $S_1$, if $\|\theta_{0,S_2}\| =O(\sigma \sqrt{\log p /n})$.
More precisely, letting $\widetilde{W} = \bX \theta_{0,S_2} + W$, the response vector $Y$ can be recast as
$Y = \bX\theta_{0,S_1} + \widetilde{W}$ and the Gauss-Lasso selector treats the small entries $\theta_{0,S_2}$
as noise. 
\end{remark}

%
%==================================================
%
\section{UCI communities and crimes data example}
\label{sec:crime}
We consider a problem about predicting the rate of violent crimes in different communities within US, based on 
other demographic attributes of the communities. 
We evaluate the performance of the Gauss-Lasso selector on the UCI communities and crimes dataset~\cite{FrankAsuncion2010}.
The dataset consists of a univariate response variable and $122$ predictive attributes for $1994$ communities.
The response variable is the total number of violent crimes per $100K$ population. Covariates are quantitative,
including e.g., the average family income, the fraction of unemployed population, and the police operating budget.
We consider a linear model as in~\eqref{eq:NoisyModel} and perform model selection using Gauss-Lasso selector and Lasso estimator.

\begin{figure}[!t]
\centering
\includegraphics*[viewport = -10 40 550 510, width =
3.5in]{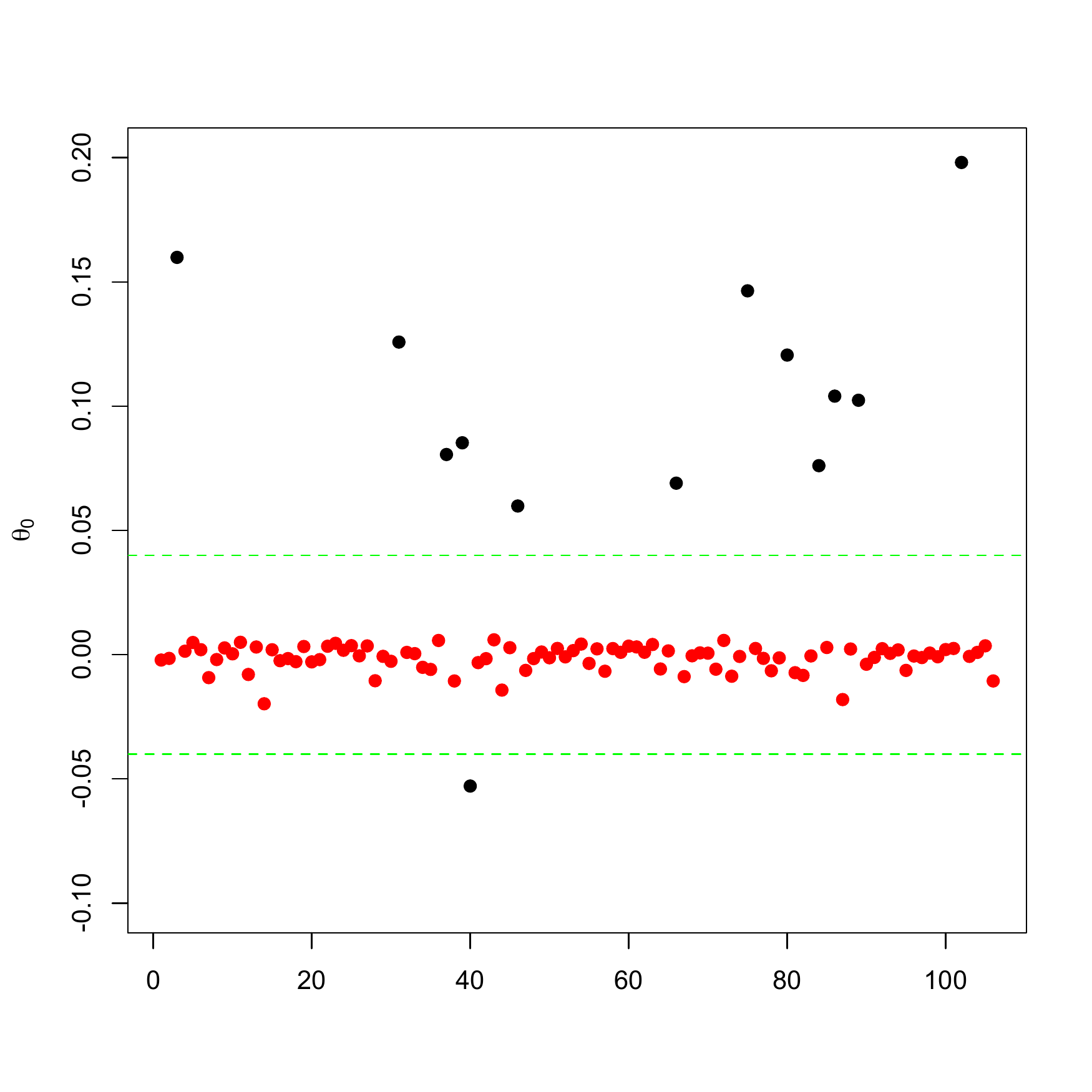}
\caption{Parameter vector $\theta_0$ for the communities dataset. The entries with magnitude larger than $0.04$ (shown in black) are treated as significant ones.}
\label{fig:true_theta}
\end{figure}

We do the following preprocessing steps: $(i)$ Each missing value is replaced by the mean of
the non-missing values of that attribute for other communities; $(ii)$ We eliminate $16$ attributes to
make the ensemble of the attribute vectors linearly independent; $(iii)$ We normalize the columns to 
have mean zero and $\ell_2$ norm $\sqrt{n}$. Thus we obtain a design matrix
$\bX_{{\rm tot}} \in \reals^{n_{{\rm tot}} \times p}$ with $n_{{\rm tot}} = 1994$ and $p = 106$.

For the sake of performance evaluation, we need to know the true model, i.e., the true significant covariates.
We let $\theta_0 = (\bX_{{\rm tot}}^\sT \bX_{{\rm tot}})^{-1} \bX_{{\rm tot}}^\sT y$ be the least square solution
obtained from the whole dataset $\bX_{{\rm tot}}$.
The entries of $\theta_0$ are shown in Fig.~\ref{fig:true_theta}. Clearly only a few of them are non negligible, corresponding
to the true model. We treat the entries with magnitude larger than
$0.04$ as truly active and the others as truly inactive.
The number of active covariates according to this criterion is $s_0 =13$.

We choose random subsamples of size $n = 85$ from the communities and normalize each column of the resulting design
matrix to have mean zero and $\ell_2$ norm $\sqrt{n}$. We use Gauss-Lasso selector and Lasso for 
model selection based on this design. Figures~\ref{fig:Gauss-Lasso} and~\ref{fig:Lasso} respectively show the solution path for Gauss-Lasso and Lasso as the parameter
$\lambda$ changes form $\lambda = 0.001$ to $\lambda = 1$. The paths corresponding to the truly active set 
are in black and the paths corresponding to the truly inactive variables are in red. At $\lambda = 1$, the solutions $\thetaGL(\lambda)$
and $\htheta^n(\lambda)$ have no active variables; for decreasing $\lambda$, each knot $\lambda_k$ marks
the entry or removal of some variables from the current active set of the Lasso solution. Therefore, the support of the Lasso solution
$T$ remains constant in between knots. Since Gauss-Lasso selector performs ordinary least squares restricted to $T$, its
coordinate paths are constant in between knots. However, the Lasso paths are linear with respect to 
$\lambda$, with changes in slope at the knots (see
e.g.,~\cite{Efron04leastangle} for a discussion).

It is clear from Figure \ref{fig:Lasso} that the Lasso support
either misses a large fraction of the truly active covariates, or
includes many false positives. 
For instance at $\lambda= 0.08$, we get
$4$ true positives out of $13$ and $4$ false positives. On the other
hand, for a smaller value of the regularization parameter,  $\lambda= 0.01$, we get $10$
true positives out of $13$ and $8$ false positives.\footnote{We treat the entries of the Lasso solution with magnitude less than $0.005$ as zero.}

If we consider on the other hand the Gauss-Lasso, any $\lambda\le 0.02$
produces a set of coefficients with a
gap between large ones, that are mostly true positives, and small
ones, that are mostly true negatives.

%{\bf [AM: Please fill in the above. Report $\lambda$'s in standard
%  decimal form, e.g. $\lambda=0.01$, etc.]}

%
\begin{figure}[!t]
\centering
\includegraphics*[viewport = 0 10 480 480, width =
3in]{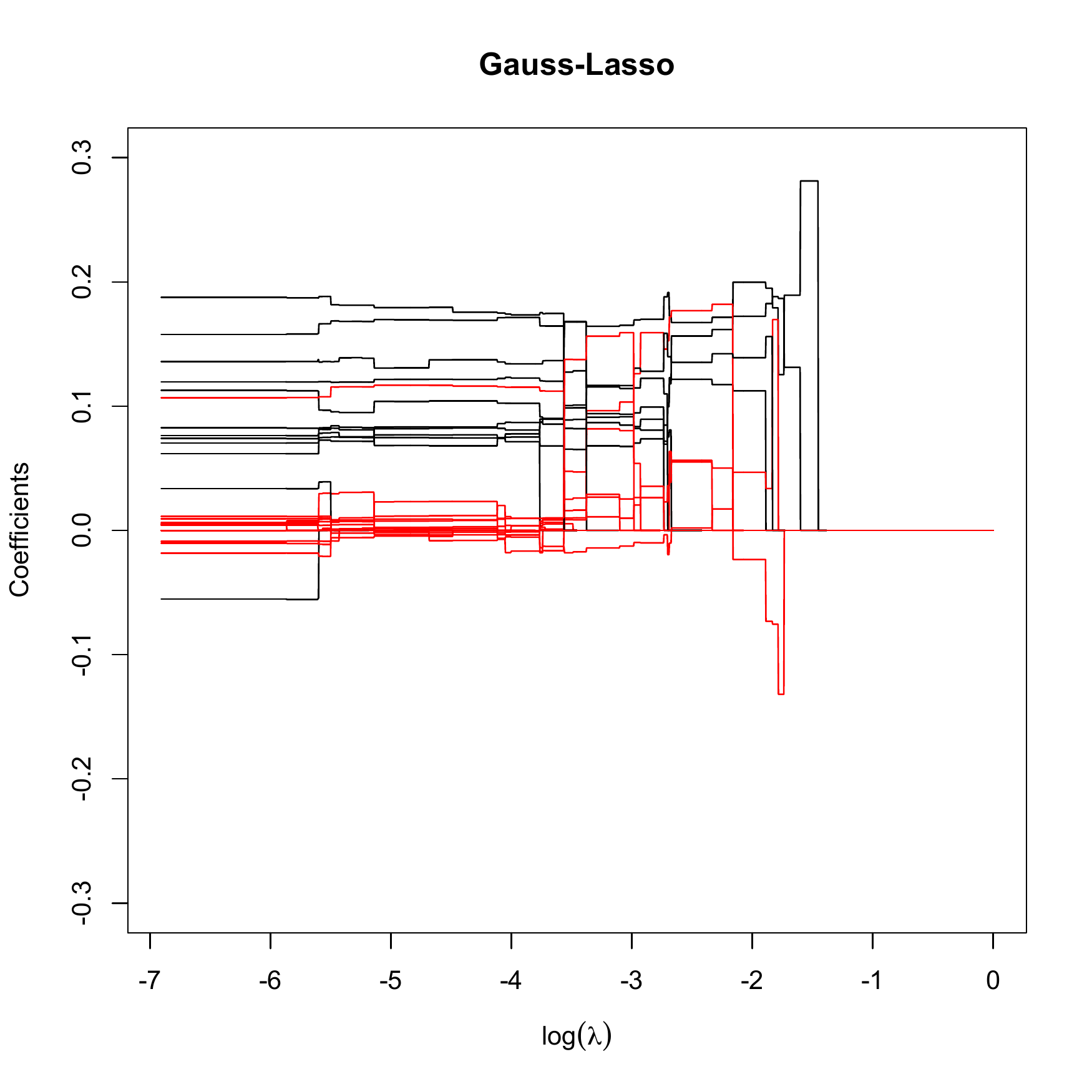}
\caption{Coordinate paths for Gauss-Lasso selector and a random subset of $n=85$ communities. The paths corresponding to the significant variables of $\theta_0$ are shown in black. The coordinate paths for Gauss-Lasso are piecewise constant.}
\label{fig:Gauss-Lasso}
\end{figure}
\begin{figure}[!t]
\centering
\includegraphics*[viewport = 0 10 480 480, width =
3in]{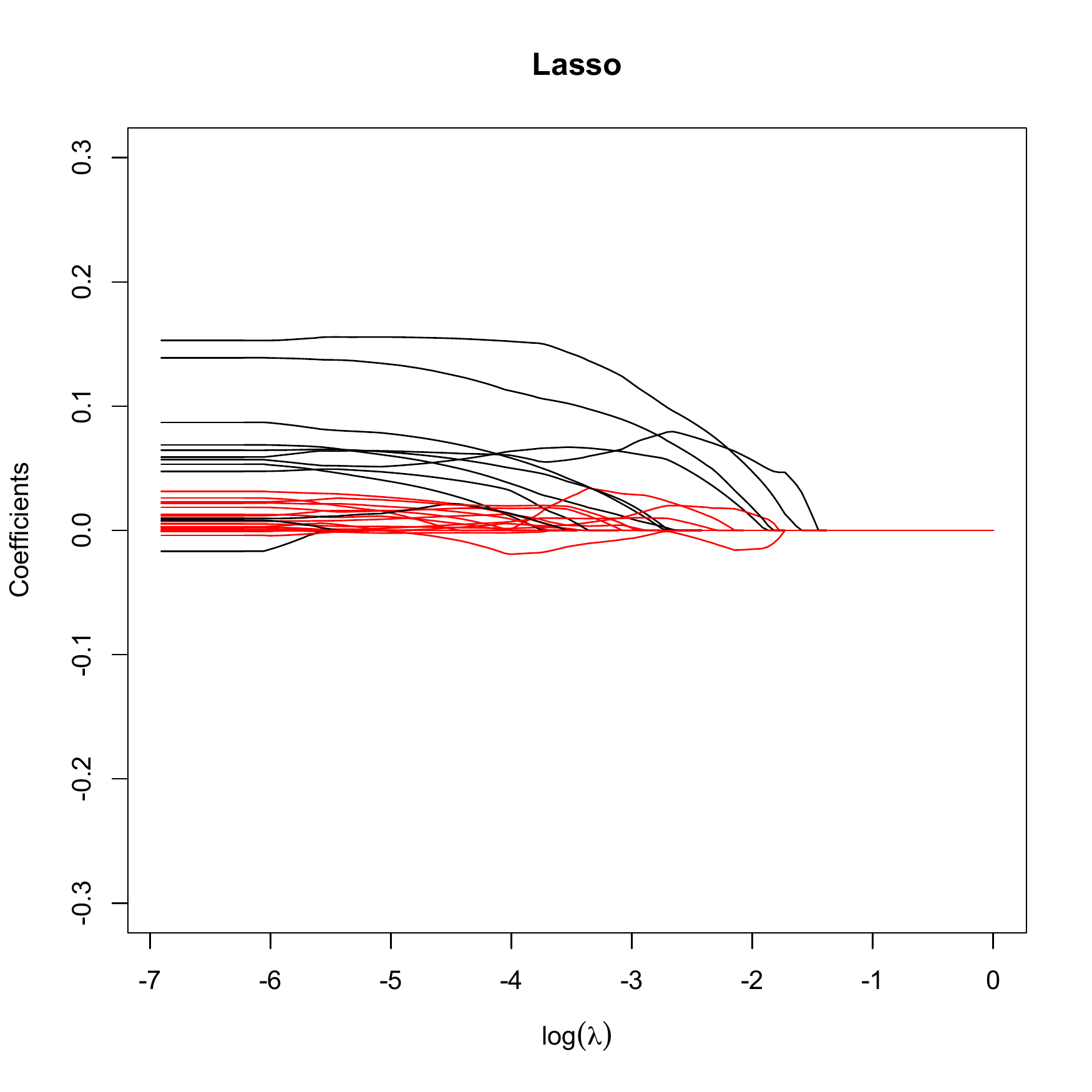}
\caption{Coordinate paths for Lasso selector and a random subset of $n=85$ communities. The paths corresponding to the significant variables of $\theta_0$ are shown in black. The coordinate paths for Lasso are piecewise linear.}
\label{fig:Lasso}
\end{figure}
%

%==================================================
%
\section{Proof of Theorems~\ref{thm:LassoSuppRecovery-DET} and \ref{thm:GLSuppRecovery-DET}}
\label{proof:two-thm}

In this section we prove Theorems~\ref{thm:LassoSuppRecovery-DET} and
\ref{thm:GLSuppRecovery-DET} using Lemmas \ref{lem:Tsize-DET} to
\ref{lem:supp-lasso-DET}.
The latter are proved in the appendices.

\subsection{Proof of Theorem~\ref{thm:LassoSuppRecovery-DET}}
\label{proof:LassoSuppRecovery-DET}

By the condition \ref{Condition:ThetaMin-DET} in the statement of the theorem, we have
\[
\lambda < \min_{i\in
  S}\left|\frac{\theta_{0,i}}{[\hSigma_{\T*,\T*}^{-1}v_{0,\T*}]_i}\right|
= \xistar\,,
\]
where the equality holds because of Lemma~\ref{lem:LimitXi0-DET}.
By Lemma~\ref{lem:LimitXi0-DET}, we know that
$\sign(\thetaZN(\lambda)) = \v*$ and that $\supp(\v*) = \T*$ contains
the true support $S$. Applying Lemma~\ref{lem:ZN-supp}, Eq.~(\ref{eq:thetaZNvT2}) and using the generalized irrepresentability assumption~\eqref{eq:GIC-DET}, we obtain
\begin{align}
&\Big\|\hSigma_{{\T*}^c,\T*}\hSigma_{\T*,\T*}^{-1}v_{0,\T*}\Big\|_{\infty}\le
1-\eta\, ,\label{eq:v*1-DET}\\
&v_{0,\T*} = \sign
\Big(\theta_{0,\T*}-\lambda \hSigma_{\T*,\T*}^{-1} v_{0,\T*}\Big)\, .\label{eq:v*2-DET}
\end{align}
Also, by Lemma~\ref{lem:supp-lasso-DET}, $\sign(\htheta^n) = \v*$ if
Eqs.~\eqref{eq:z1-DET} and~\eqref{eq:z2-DET} hold with $z = \v*$ and
$T = \T*$, namely, if
\begin{align}
\Big\|
\hSigma_{{\T*}^c,\T*}&\hSigma_{\T*,\T*}^{-1} v_{0,\T*} + \frac{1}{\lambda} (\hr_{{\T*}^c} - \hSigma_{{\T*}^c,\T*}\hSigma_{\T*,\T*}^{-1}\hr_{\T*})
\Big\|_{\infty}\le
1\, , \label{eq:z1->v1-DET}\\
&v_{0,\T*} = \sign\Big(\theta_{0,T} - \hSigma^{-1}_{\T*,\T*} (\lambda v_{0,\T*} - \hr_{\T*}) \Big)\,.\label{eq:z2->v2-DET}
\end{align}
In the sequel, we show that these equations are satisfied, with probability lower bounded as per Eq.~(\ref{MainThm1:ProbabilityEstimate}).

We begin with proving Eq.~\eqref{eq:z1->v1-DET}. Let $\term = (1/\lambda) (\hr_{{\T*}^c} - \hSigma_{{\T*}^c,\T*}\hSigma_{\T*,\T*}^{-1}\hr_{\T*})$. We need to show that $\|\term\|_{\infty} \le \eta $. Plugging for $\hr$, we get
$\term \equiv \bX_{\T*^c} \Pi_{\bX^\perp_{\T*}} W/(n\lambda)$, where $\Pi_{\bX^\perp_{\T*}} = \id - \bX_{\T*} (\bX_{\T*}^\sT \bX_{\T*})^{-1} \bX_{\T*}^\sT$ is the orthogonal projection  onto the orthogonal complement of the column space of $\bX_{\T*}$.
Since $W \sim \normal(0,\sigma^2 \id_{n\times n})$, the variable $\term_j = x_j^\sT \Pi_{\bX^\perp_{\T*}} W /(n \lambda)$ is normal with variance at most
\[
\Big(\frac{\sigma}{n\lambda}\Big)^2\|\Pi_{\bX^\perp_{\T*}} x_j\|_2^2 \le
\Big(\frac{\sigma}{n\lambda}\Big)^2\|x_j\|_2^2\le \frac{\sigma^2}{n\lambda^2}\,,
\]
where we used  the fact that 
$\|x_j\|^2 \le n$, as $\hSigma_{i,i} \le 1$.
By the  Gaussian tail bound with union bound over $j\in \T*^c$, we obtain
\begin{eqnarray}\label{eq:term1-DET}
\prob(\|\term\|_{\infty} \le \eta ) \ge 1 - 2pe^{-\frac{n\lambda^2\eta^2}{2\sigma^2}} = 1- 2p^{1-c_1}\,.
\end{eqnarray}

We next prove Eq.~\eqref{eq:z2->v2-DET}.
Given Eq.~\eqref{eq:v*2-DET}, we need to show 
\begin{align*}
\sign\Big(\theta_{0,\T*} - \lambda \hSigma^{-1}_{\T*,\T*} v_{0,\T*}\Big) = \sign\Big(\theta_{0,\T*} - \hSigma^{-1}_{\T*,\T*} (\lambda v_{0,\T*} - \hr_{\T*})\Big)\,.
\end{align*}
Let $ u \equiv \theta_{0,\T*} - \lambda \hSigma^{-1}_{\T*,\T*}
v_{0,\T*}$, and $\hu \equiv \theta_{0,\T*} -
\hSigma^{-1}_{\T*,\T*} (\lambda v_{0,\T*} - \hr_{\T*})$.

By condition \emph{\ref{Condition:ThetaMin}}, we have, for all
$i\in S$, $|u_i|\ge |\theta_{0,i}| - \lambda |[\hSigma^{-1}_{\T*,\T*}
v_{0,\T*}]_i|\ge c_2\lambda$.
Further, for all $i\in\T*\setminus S$, we have $|u_i| = \lambda |[\hSigma^{-1}_{\T*,\T*}
v_{0,\T*}]_i|\ge c_2\lambda$.
Summarizing, for all $i \in \T*$, we have
$|u_i| \ge c_2\lambda$.
We will show that  $\|u - \hu\|_\infty = \|\hSigma^{-1}_{\T*,\T*} \hr_{\T*}\|_\infty < c_2\lambda$, with high probability, thus implying
$\sign(u_{\T*})=\sign(\hu_{\T*})$ as desired.
\begin{lemma}\label{lem:term2-DET}
The following holds true.
\begin{align}\label{eq:term2-DET}
\prob\Big( \|\hSigma_{\T*,\T*}^{-1} \hr_{\T*}\|_\infty \ge \sigma \sqrt{\frac{2c_1\log p}{n}}\, \|\hSigma^{-1}_{\T*,\T*}\|_2^{1/2} \Big) \le 2p^{1-c_1}\,.
\end{align}
\end{lemma}
Lemma~\ref{lem:term2-DET} is proved by noting that conditioned on $\bX_{\T*}$, $\hSigma_{\T*,\T*}^{-1} \hr_{\T*}$ is a Gaussian vector and then applying standard tail bound inequality.
The details are deferred to Section~\ref{proof:term2-DET}.

Using Lemma~\ref{lem:term2-DET} and the assumption $\eta \le c_2 \sqrt{C_{\min}}$, we get $\|u-\hu\|_{\infty} < c_2 \lambda$,
with probability at least $1 - 2p^{1-c_1}$. 

Putting all this together, Eqs.~\eqref{eq:z1->v1-DET} and~\eqref{eq:z2->v2-DET} hold 
simultaneously, with probability at least $1 - 4p^{1-c_1}$. This implies the thesis.

%==================================================
%
\subsection{Proof of Theorem~\ref{thm:GLSuppRecovery-DET}}
\label{proof:GLSuppRecovery-DET}

Recall that $T = \supp(\htheta^{n})$. On the event $\event \equiv \{T = \T*\}$, we have
\[
\thetaGL_{T} = (\bX_T^\sT \bX_T)^{-1} \bX_T^\sT (\bX_T \theta_{0,T} + W)
= \theta_{0,T} +  (\bX_T^\sT \bX_T)^{-1} \bX_T^\sT W\,,
\]
where the first equality holds since $T = \T* \supseteq S$ and thus $\theta_{0,T^c} = 0$.
Further note that $\thetaGL_i - \theta_{0,i}$, for $i\in T$, is a zero mean Gaussian vector
with variance 
\[
\sigma^2 \|e_i^\sT (\bX_{T}^\sT \bX_{T})^{-1} \bX_{T}^\sT\|^2
\le \sigma^2 \|\hSigma^{-1}_{T,T}\|_2/n \le \sigma^2/(n C_{\min})\,.
\]
Using tail bound inequality along with union bounding over $i \in [p]$, we get
\[
\prob\Big(\|\thetaGL_{T} - \theta_{0,T}\|_\infty \ge \mu; \event\Big) \le 2e^{-nC_{\min} \mu^2/2\sigma^2}\,.
\]
Also, under the assumptions of Theorem~\ref{thm:LassoSuppRecovery-DET}, $\prob(\event) \ge 1- 4 p^{1-c_1}$. Hence
\[
\prob\Big(\|\thetaGL_{T} - \theta_{0,T}\|_\infty \ge \mu\Big)
\le \prob\Big(\|\thetaGL_{T} - \theta_{0,T}\|_\infty \ge \mu; \event\Big) + \prob(\event^c)
\le 2e^{-nC_{\min} \mu^2/2\sigma^2} + 4p^{1-c_1}\,.
\]
Since $\thetaGL_{T^c} = \theta_{0,T^c} = 0$, we get $\|\thetaGL - \theta_0\|_\infty <\mu$, with probability at least
$1- 4p^{1-c_1} - 2e^{-nC_{\min} \mu^2/2\sigma^2}$.

Moreover, if $\|\thetaGL - \theta_0\| < \theta_{\min}/2$, then $|\thetaGL_i| > \theta_{\min}/2$ for $i \in S$
and $|\thetaGL_i| < \theta_{\min}/2$, for $i \in S^c$. Hence, the $s_0$ top entries of $\thetaGL$ (in modulus), 
returned by the Gauss-Lasso selector,
correspond to the true support $S$. Therefore,
\begin{align*}
\prob(\hat{S} = S) &\ge \prob(\|\thetaGL - \theta_0\|_\infty < \theta_{\min}/2)\\
& \ge 1- 4p^{1-c_1} - 2pe^{-n C_{\min} \theta_{\min}^2/8\sigma^2}
\ge 1 - 6p^{1-c_1/4}\,,
\end{align*}
where the last inequality follows from the facts $\theta_{\min} \ge c_2 \lambda$, and $ \eta \le c_2 \sqrt{C_{\min}}$. 
%Recall that, for $i\in S$,  $|\theta_{0,i}|\ge c_2\lambda$ by assumption
%\ref{Condition:ThetaMin-DET} of Theorem
%\ref{thm:LassoSuppRecovery-DET}.
%Using $\mu = c_2\lambda/2$ in the above bound, we conclude that, with
%probability at least $1-4p^{1-c_1}-2p^{1-c_1/4}$, we have
%$|\thetaGL_i|>c_2\lambda/2$ for $i\in S$ and
%$|\thetaGL_i|<c_2\lambda/2$ for $i\in T_0\setminus S$.
%This implies our claim.
%===================================================
%
\section{Proof of Theorems~\ref{thm:LassoSuppRecovery} and \ref{thm:GLSuppRecovery}}
\label{proof:LassoSuppRecovery}
\addtocontents{toc}{\protect\setcounter{tocdepth}{1}}

By the condition \ref{Condition:ThetaMin} in the statement of the theorem, we have
\[
\lambda \le \frac{2}{3}\min_{i\in
  S}\left|\frac{\theta_{0,i}}{[\Sigma_{\T*,\T*}^{-1}v_{0,\T*}]_i}\right|
< \xistar\,,
\]
where the second inequality holds because of Lemma~\ref{lem:LimitXi0}.
Therefore, as a result of Lemma~\ref{lem:LimitXi0}, we have $\sign(\htheta^\infty(\lambda)) = \v*$ and that $\supp(\v*) = \T*$ contains the true support $S$. Applying Lemma~\ref{lem:infinity-supp} and using the generalized irrepresentability assumption, we have
\begin{align}
&\Big\|\Sigma_{{\T*}^c,\T*}\Sigma_{\T*,\T*}^{-1}v_{0,\T*}\Big\|_{\infty}\le
1-\eta\, ,\label{eq:v*1}\\
&v_{0,\T*} = \sign
\Big(\theta_{0,\T*}-\lambda \Sigma_{\T*,\T*}^{-1} v_{0,\T*}\Big)\, .\label{eq:v*2}
\end{align}
Moreover, by Lemma~\ref{lem:supp-lasso-DET}, $\sign(\htheta^n) = \v*$ if Eqs.~\eqref{eq:z1-DET} and~\eqref{eq:z2-DET} hold with $z = \v*$ and $T = \T*$, namely,
\begin{align}
\Big\|
\hSigma_{{\T*}^c,\T*}&\hSigma_{\T*,\T*}^{-1} v_{0,\T*} + \frac{1}{\lambda} (\hr_{{\T*}^c} - \hSigma_{{\T*}^c,\T*}\hSigma_{\T*,\T*}^{-1}\hr_{\T*})
\Big\|_{\infty}\le
1\, , \label{eq:z1->v1}\\
&v_{0,\T*} = \sign\Big(\theta_{0,T} - \hSigma^{-1}_{\T*,\T*} (\lambda v_{0,\T*} - \hr_{\T*}) \Big)\,.\label{eq:z2->v2}
\end{align}
The rest of the proof is devoted to show the validity of these
equations, with probability lower bounded as per Eq.~(\ref{MainThm:ProbabilityEstimate}).

%***************************************************************************************

\subsection{Proof of Eq.~(\ref{eq:z1->v1})}

It is immediate to see that Eq.~\eqref{eq:z1->v1} holds if the followings hold true:
\begin{align}
&\term_1 \equiv \big\|\hSigma_{{\T*}^c,\T*} \hSigma^{-1}_{\T*,\T*} v_{0,\T*} \big\|_{\infty} \le 1-\frac{\eta}{2}\,,\label{eq:mother-1}\\
&\term_2 \equiv \frac{1}{\lambda} \big\|\hr_{{\T*}^c} - \hSigma_{{\T*}^c,\T*} \hSigma^{-1}_{\T*,\T*} \hr_{\T*} \big\|_{\infty} \le \frac{\eta}{2}\,.\label{eq:mother-2}
\end{align}

In order to prove inequalities~\eqref{eq:mother-1} and \eqref{eq:mother-2}, it is useful to recall the following proposition from random matrix theory.
\begin{propo}[\cite{Szarek:survey,Wainwright2009LASSO,Vershynin-CS}]\label{pro:sigmadiff-opnorm}
For $k\le n$, let $\bX \in \reals^{n \times k}$ be a random matrix with i.i.d rows drawn from $\normal(0,\Sigma)$.
Then the following hold true for all $t\ge 1$ and $\tau \equiv 2(\sqrt{\frac{k}{n}} + t) + (\sqrt{\frac{k}{n}} + t)^2$\,.
\begin{itemize}
\item[(a)] If $\Sigma$ has maximum eigenvalue $\smax < \infty$, then
\begin{align*}
\prob\bigg(\|\frac{1}{n} \bX^\sT \bX- \Sigma\|_2 \ge \smax\, \tau \bigg)\le 2 e^{-nt^2/2}\,.
\end{align*}
\item[(b)] If $\Sigma$ has minimum eigenvalue $\smin > 0$, then
\begin{align*}
\prob\bigg(\|(\frac{1}{n} \bX^\sT \bX)^{-1}- \Sigma^{-1}\|_2 \ge \smin^{-1}\, \tau \bigg)\le 2 e^{-nt^2/2}\,.
\end{align*}
\end{itemize}
\end{propo}
We consider the particular choice of $t= \sqrt{k/n}$ which is useful for future reference. Since $k/n \le 1$, we get $\tau \le 8\sqrt{k/n}$ and therefore the specialized version
of Proposition~\ref{pro:sigmadiff-opnorm} reads:
\begin{eqnarray}
\prob\bigg(\|\frac{1}{n} \bX^\sT \bX- \Sigma\|_2 \ge 8\sqrt{k/n}\, \smax \bigg)&\le &2 e^{-k/2}\,, \label{eq:diff-opt-spec1}\\
\prob\bigg(\|(\frac{1}{n} \bX^\sT \bX)^{-1}- \Sigma^{-1}\|_2 \ge 8 \sqrt{k/n}\, \smin^{-1} \bigg)&\le& 2 e^{-k/2}\,. \label{eq:diff-opt-spec2}
\end{eqnarray}
We define the event $\event_1$ as
\begin{align*}
%\event_1 &\equiv \bigg\{ \|\hSigma_{\T*,\T*} - \Sigma_{\T*,\T*}\|_2 \le 8 \sqrt{t_0/n}\,C_{\max} \bigg\}\,,\\
\event_1 &\equiv \bigg\{ \|(\hSigma_{\T*,\T*})^{-1} - \Sigma_{\T*,\T*}^{-1}\|_2 \le 8 \sqrt{t_0/n}\,C_{\min}^{-1} \bigg\}\, .
\end{align*}
Applying Eqs.~\eqref{eq:diff-opt-spec1},~\eqref{eq:diff-opt-spec2} to
$\bX_{\T*}$, we conclude that
\begin{align}
\prob(\event_1^c) \le 2e^{-t_0/2}\, .\label{eq:Event1Bond}
\end{align}
We now have in place all we need to bound the terms $\term_1$ and $\term_2$.

%%%%%%%%%%%%%%%%%%%%%%%

\subsubsection{Bounding $\term_1$}

To bound $\term_1$, we employ similar techniques to those used in~\cite[Theorem 3]{Wainwright2009LASSO} to verify strict dual feasibility.
The argument in~\cite{Wainwright2009LASSO} works under the irrepresentability condition (see Eq.~(26) therein) and we modify it to apply to the current setting, i.e.,
 the generalized irrepresentability condition.

We begin by conditioning on $\bX_{\T*}$. For $j\in {\T*}^c$, $x_j$ is a zero mean Gaussian vector and we can decompose
it into a linear correlated part plus an uncorrelated part as
\[
x_j^\sT = \Sigma_{j,\T*} \Sigma_{\T*,\T*}^{-1} \bX_{\T*}^\sT + \epsilon_j^\sT\,,
\]  
where $\epsilon_j \in \reals^n$ has i.i.d. entries distributed as $\epsilon_{ji} \sim \normal(0, \Sigma_{j,j} - \Sigma_{j,\T*} \Sigma_{\T*,\T*}^{-1} \Sigma_{\T*,j})$.

Letting $u = \hSigma_{{\T*}^c,\T*} \hSigma^{-1}_{\T*,\T*} v_{0,\T*} $, we write
\begin{align}
u_j &= x_j^\sT \bX_{\T*}(\bX_{\T*}^\sT \bX_{\T*})^{-1} v_{0,\T*}  \nonumber \\
&=  \Sigma_{j,\T*} (\Sigma_{\T*,\T*})^{-1} v_{0,\T*} 
+ \epsilon_j^\sT \bX_{\T*} (\bX_{\T*}^\sT \bX_{\T*})^{-1} v_{0,\T*}\,. \label{eq:uj}
\end{align}
The first term is bounded as $|\Sigma_{j,\T*} (\Sigma_{\T*,\T*})^{-1} v_{0,\T*}| \le 1-\eta$ as per Eq.~\eqref{eq:v*1}.
Let $m_j =  \epsilon_j^\sT \bX_{\T*}(\bX_{\T*}^\sT \bX_{\T*})^{-1}v_{0,\T*}$. Since $\Var(\epsilon_{ji}) \le \Sigma_{j,j} \le 1$, conditioned on $\bX_{\T*}$, $m_j$ is zero mean Gaussian with variance at most
\begin{align}
\Var(m_j) &\le  \|\bX_{\T*}(\bX_{\T*}^\sT \bX_{\T*})^{-1} v_{0,\T*} \|_2^2 \nonumber\\
&\le \frac{1}{n} v_{0,\T*}^\sT \Big(\frac{\bX_{\T*}^\sT \bX_{\T*}}{n}\Big)^{-1} v_{0,\T*} \nonumber \\
&\le \frac{1}{n} \|\hSigma_{\T*,\T*}^{-1}\|_2\, \|v_{0,\T*}\|^2 \,.\label{eq:var-mj}
\end{align}
Under the event $\event_1$, we have
\begin{align}\label{eq:event2-cons}
\|\hSigma^{-1}_{\T*,\T*}\|_2 \le \|\Sigma_{\T*,\T*}^{-1}\|_2 + \|\hSigma_{\T*,\T*}^{-1} - \Sigma_{\T*,\T*}^{-1}\|_2 \le (1+8\sqrt{t_0/n})\, C_{\min}^{-1} \le 9C_{\min}^{-1}\,,
\end{align}
and hence, $\Var(m_j) \le 9 t_0 /(n C_{\min})$. 
We now define the event $\event$ as 
\[
\event \equiv \bigg\{\max_{j\in T^c} |m_j| \ge \sqrt{\frac{18c_1\, t_0 \log p}{n\, C_{\min}}} \bigg\}\,.
\]
By the total probability rule, we have
\[
\prob(\event) \le \prob(\event; \event_1 ) + \prob(\event_1^c)\,.
\]
Using Gaussian tail bound and union bounding over $j\in {\T*}^c$, we obtain $\prob(\event; \event_1) \le 2p^{1-c_1}$.
Using the bound $\prob(\event_1^c) \le 2e^{-t_0/2}$, we arrive at:

%%
%\begin{align}\label{eq:T2-3}
%\prob \left( \max_{j \in T^c} |m_j| > 6\sqrt{\frac{c_1\, t_0 \log p}{n\,C_{\min}} \bigg(1+\frac{4c_1 \sigma^2 \log p}{n \lambda^2}}\bigg)\right) 
%\le 2(p^{1-c_1} +  e^{-\frac{t_0}{2}}) + pe^{-\frac{n}{100}}\,.
%\end{align} 
%%
%
\begin{align}\label{eq:T2-3}
\prob \left( \max_{j \in T^c} |m_j| > \sqrt{\frac{18 c_1\, t_0 \log p}{n\,C_{\min}}} \right) 
\le 2p^{1-c_1} +  2e^{-\frac{t_0}{2}}\,.
\end{align} 
Using this, together with Eq.~\eqref{eq:v*1},  in Eq.~\eqref{eq:uj}, we obtain that
the following holds true with probability at least $1-2p^{1-c_1} - 2e^{-t_0/2}$:
\begin{align}
\term_1 \le
1-\eta + \sqrt{ \frac{18 c_1\, t_0 \log p}{n\,C_{\min}}}\,. \label{eq:mother-1-proof}
\end{align}
It is easy to check that the this implies $\term_1<1-\eta/2$, for
$\lambda$ as claimed in Eq.~(\ref{eq:lambda_val}) provided $n\ge M_1t_0\log p$.

%%%%%%%%%%%%%%%%%%%%%%%

\subsubsection{Bounding $\term_2$}
We bound $\term_2$ by the same technique used in proving Eq.~\eqref{eq:z1->v1-DET}.
Let $m= (1/\lambda) (\hr_{{\T*}^c} - \hSigma_{{\T*}^c,\T*}\hSigma_{\T*,\T*}^{-1}\hr_{\T*})$. 
Plugging for $\hr$, we get
$m \equiv \bX_{\T*^c} \Pi_{\bX^\perp_{\T*}} W/(n\lambda)$.
Since $W \sim \normal(0,\sigma^2 \id_{n\times n})$, conditioned on
$\bX$,  the variable $m_j = x_j^\sT \Pi_{\bX^\perp_{\T*}} W /(n \lambda)$ is normal with variance at most
\[
(\frac{\sigma}{n\lambda})^2 \|\Pi_{\bX^\perp_{\T*}} x_j\|_2^2 \le (\frac{\sigma}{n\lambda})^2 \|x_j\|^2\,,
\]
where we used the contraction property of orthogonal projections.
Now, define the event $\event$ as follows.
\[
\event \equiv \bigg\{\|x_j\|^2 < 2n, \forall j\in [p] \bigg\}\,.
\]
Note that $\|x_j\|^2 \deq \Sigma_{j,j} Z$, where $Z$ is a chi-squared random variable with $n$ degrees of 
freedom. Using the standard chi-squared
tail bounds~\cite{Johnstone-chi2}, for a fixed $j$, we have $\|x_j\|^2 < 2\Sigma_{j,j}\, n \le 2n$, 
with probability at least $1 - e^{-n/10}$. Union bounding over $j \in [p]$, we obtain
$\prob(\event^c) \le p e^{-n/10}$. 

Under the event $\event$, we have $\Var(m_j) \le 2\sigma^2/(n\lambda^2)$. Employing the standard Gaussian tail bound along with union bounding over $j\in \T*^c$, we obtain
\begin{eqnarray}
\prob(\term_2 \ge \eta/2;\, \event ) \le 2pe^{-\frac{n\lambda^2\eta^2}{16\sigma^2}} = 2p^{1-c_1}\,.
\end{eqnarray}
Hence, 
\begin{eqnarray}
\prob(\term_2 \ge \eta/2) \le \prob(\term_2 \ge \eta/2;\, \event ) + \prob(\event^c) \le
2p^{1-c_1} + p e^{-\frac{n}{10}}  \,. \label{eq:mother-2-proof}
\end{eqnarray}
% 
%======================================================================

\subsection{Proof of Eq.~(\ref{eq:z2->v2})}

We next prove Eq.~\eqref{eq:z2->v2}. Given Eq.~\eqref{eq:v*2}, we need to show 
\begin{align*}
\sign\Big(\theta_{0,\T*} - \lambda \Sigma^{-1}_{\T*,\T*} v_{0,\T*}\Big) = \sign\Big(\theta_{0,\T*} - \hSigma^{-1}_{\T*,\T*} (\lambda v_{0,\T*} - \hr_{\T*})\Big)\,.
\end{align*}
Let $ u \equiv \theta_{0,\T*} - \lambda \Sigma^{-1}_{\T*,\T*}
v_{0,\T*}$, and $\hu \equiv \theta_{0,\T*} -
\hSigma^{-1}_{\T*,\T*} (\lambda v_{0,\T*} - \hr_{\T*})$.

By condition \ref{Condition:ThetaMin}, we have, for all
$i\in S$, $|u_i|\ge |\theta_{0,i}| - \lambda |[\Sigma^{-1}_{\T*,\T*}
v_{0,\T*}]_i|\ge c_2\lambda+(1/2) \lambda |[\Sigma^{-1}_{\T*,\T*}
v_{0,\T*}]_i|$. Further, for all $i\in\T*\setminus S$, we have $|u_i| = \lambda |[\Sigma^{-1}_{\T*,\T*}
v_{0,\T*}]_i|\ge c_2\lambda+(1/2) \lambda |[\Sigma^{-1}_{\T*,\T*}
v_{0,\T*}]_i|$. Summarizing, for all $i\in \T*$, we have
\begin{align*}
|u_i| \ge c_2\lambda +\frac{1}{2}\lambda |[\Sigma^{-1}_{\T*,\T*}
v_{0,\T*}]_i|\, .
\end{align*}
We will show that  $|u_i - \hu_i| <  c_2\lambda +(1/2)\lambda |[\Sigma^{-1}_{\T*,\T*}
v_{0,\T*}]_i|$ for all $i\in\T*$, with high probability, thus implying
$\sign(u_{\T*})=\sign(\hu_{\T*})$ as desired.
Since $|u_i - \hu_i| \le \lambda |[(\hSigma_{\T*,\T*}^{-1} - \Sigma_{\T*,\T*}^{-1}) v_{0,\T*} ]_i|+|[\hSigma^{-1}_{\T*,\T*} \hr_{\T*}]_i|$, it suffices
to show that
\begin{align}
&\term_3(i) \equiv \lambda |[(\hSigma_{\T*,\T*}^{-1} - \Sigma_{\T*,\T*}^{-1}) v_{0,\T*} ]_i\big| <
\frac{1}{2}\lambda |[\Sigma^{-1}_{\T*,\T*}
v_{0,\T*}]_i| \,\;\;\;\;\;\;\mbox{for all } i\in\T*\label{eq:mother-3},\\
&\term_4 \equiv \|\hSigma_{\T*,\T*}^{-1} \hr_{\T*}\|_{\infty} <
c_2\lambda\,.\label{eq:mother-4}
\end{align}
In the sequel, we provide probabilistic bounds on $\term_3(i)$ and $\term_4$.
%%%%%%%%%%%%%%%%

\subsubsection{Bounding $\term_3(i)$}
\begin{lemma}\label{lem:mother-3}
Under the assumptions of Theorem~\ref{thm:LassoSuppRecovery}, for any
$c' > 1$, $t_0\ge 4$, we have
\begin{align*}
\prob\left\{ \exists i\in\T* \;\;\mbox{{\rm s.t. }}\big|[(\hSigma_{\T*,\T*}^{-1} - \Sigma_{\T*,\T*}^{-1}) v_{0,\T*}]_i\big| \ge 16\sqrt{\frac{c'c_*\, t_0\log p}{n}} \big|[\Sigma_{\T*,\T*}^{-1} v_{0,\T*}]_i\big| \right\}\le
 2e^{-\frac{t_0}{2}}+2p^{1-c'}\,,
\end{align*}
where $c_*\equiv (c_2C_{\min})^{-2}$.
\end{lemma} 
The proof of Lemma~\ref{lem:mother-3} is presented in Section~\ref{proof:mother-3}.

Applying this lemma, with probability at least $1-2e^{-t_0/2} -
2p^{1-c_1}$, we have $\term_3(i) <(1/2)\lambda |[\Sigma^{-1}_{\T*,\T*}
v_{0,\T*}]_i|$ provided 
\begin{align*}
16\sqrt{\frac{c_1c_*\, t_0\log p}{n}} \le \frac{1}{2}\, .
\end{align*}
i.e., for $n\ge M_3t_0\log p$.
%
%%%%%%%%%%%%%%%%%%%%%%%%

\subsubsection{Bounding $\term_4$}

\begin{lemma}\label{lem:term4}
The following holds true.
\begin{align}\label{eq:mother-4-proof}
\prob\bigg(\term_4 \le 3\sigma\sqrt{\frac{2c_1\log p}{n\, C_{\min}}}\bigg) \ge 1 - 2e^{-\frac{t_0}{2}} - 2p^{1-c_1}\,.
\end{align}
\end{lemma}
Lemma~\ref{lem:term4} is proved in Section~\ref{proof:term4}. 

From the last lemma, it follows that  Eq.~\eqref{eq:mother-4}
holds with probability at least  $1 - 2e^{-\frac{t_0}{2}} - 2p^{1-c_1}$,
provided
\begin{align*}
3\sigma\sqrt{\frac{2c_1\log p}{n\, C_{\min}}}\le
c_2\lambda\, .
\end{align*}
Choosing $\lambda$ as per Eq.~\eqref{eq:lambda_val},
the latter is easily shown to follow from $\eta\le
c_2\sqrt{C_{\min}}$.

\subsection{Summary: Proof of Theorem \ref{thm:LassoSuppRecovery}}
 
 Now combining the bounds  on $\term_1$,\dots $\term_4$, we get that for $n \ge \max(M_1, M_3)\, t_0 \log p$, 
 Eqs.~\eqref{eq:z1->v1} and~\eqref{eq:z2->v2} hold simultaneously, 
 with probability at least $1 - p e^{-n/10} - 6e^{-t_0/2} - 8p^{1-c_1}$. This implies $\sign(\htheta^n(\lambda)) = v_0$.
\subsection{Proof of Theorem \ref{thm:GLSuppRecovery}}
\label{proof:GLSuppRecovery}
Note that the matrix $\bX_{\T*}$ is a random Gaussian matrix 
with rows drawn independently form $\normal(0,\Sigma_{\T*,\T*})$
(recall that $\T*$ is a deterministic set determined
by the population-level problem). Therefore,
$\|\hSigma_{\T*,\T*}^{-1}\|_2 \le 9 \|\Sigma_{\T*,\T*}^{-1}\|_2 \le 9
C_{\min}^{-1}$.
Using Theorem \ref{thm:LassoSuppRecovery} to bound the probability
that $T\neq \T*$,  the proof proceeds along the same lines as the proof of Theorem~\ref{thm:GLSuppRecovery-DET}. 

\addtocontents{toc}{\protect\setcounter{tocdepth}{1}}

\section*{Acknowledgements}
A.J. is supported by a Caroline and Fabian
Pease Stanford Graduate Fellowship.
This work was partially supported by the NSF CAREER award CCF-0743978, the NSF
grant DMS-0806211, and the grants AFOSR/DARPA FA9550-12-1-0411 and FA9550-13-1-0036.
%==========================================================
%
\appendix

\section{Proof of technical lemmas}
\addtocontents{toc}{\protect\setcounter{tocdepth}{1}}
\subsection{Proof of Lemma~\ref{lem:Tsize-DET}}
\label{proof:Tsize-DET}

By a change of variables, it is easy to see that
$\thetaZN(\xi) =\theta_0+\xi\,\hu(\xi)$, where $\hu(\xi) =
\arg\min_{u\in\reals^p} F(u;\xi)$ and
\begin{align*}
F(u;\xi) \equiv \frac{1}{2}\<u,\hSigma u\> + \|u_{S^c}\|_1 + 
\Big(\|\xi^{-1}\theta_{0,S}+u_S\|_1-\|\xi^{-1}\theta_{0,S}\|_1\Big)\, .
\end{align*}

The rest of the proof is analogous to an argument in \cite{BickelEtAl}.
Since, by definition, $F(\hu;\xi)\le F(0;\xi)$, we have
\begin{align}
\frac{1}{2}\<\hu,\hSigma \hu\>+ \|\hu_{S^c}\|_1-\|\hu_{S}\|_1\le 0\label{eq:BoundFstar}
\end{align}
and hence $\|\hu_{S^c}\|_1\le\|\hu_{S}\|_1$. Using the definition of
$\hkappa$, with $J =S$, $c_0=1$, we have
\begin{align*}
0 &\ge\frac{1}{2}\hkappa(s_0,1)\|\hu\|_2^2+
\|\hu_{S^c}\|_1-\|\hu_{S}\|_1\\
&\ge \frac{1}{2}\hkappa(s_0,1) \|\hu_{S}\|_2^2-\|\hu_{S}\|_1\, ,
\end{align*}
and since $\|\hu_{S}\|_2^2\ge \|\hu_{S}\|_1^2/s_0$, we deduce that
\begin{align*}
\|\hu_{S}\|_1\le \frac{2s_0}{\hkappa(s_0,1)}\, .
\end{align*}
By Eq.~(\ref{eq:BoundFstar}), this implies in turn 
\begin{align}\label{eq:BoundSigma}
\<\hu,\hSigma \hu\>\le  \frac{4s_0}{\hkappa(s_0,1)} \, .
\end{align}

Now, consider the stationarity conditions of $F$. These imply
\begin{align*}
(\hSigma \hu)_i &= -\sign(\hu_{i})\,,\;\;\;\;\;\;\;\mbox{for
  $i\in T\setminus S$}.
\end{align*}
We therefore have
\begin{align*}
|T\setminus  S|\le \sum_{i\in T\setminus S}(\hSigma \hu)_i^2\le \|\hSigma \hu\|_2^2\le
\|\hSigma\|_2\<\hu,\hSigma \hu\>\, ,
\end{align*}
and our claim follows by substituting Eq.~(\ref{eq:BoundSigma}) in the
latter equation.

%%%%%%%%%%%%%%%%%%%%%%%%%%%%%%%%%%%%%%%%%
\subsection{Proof of Lemma~\ref{lem:LimitXi0-DET}}
\label{proof:LimitXi0-DET}
By a change of variables, it is easy to see that
$\thetaZN(\xi) =\theta_0+\xi\,\hu(\xi)$, where $\hu(\xi) =
\arg\min_{u\in\reals^p} F(u;\xi)$ and
\begin{align*}
F(u;\xi) \equiv \frac{1}{2}\<u,\hSigma u\> + \|u_{S^c}\|_1 + 
\Big(\|\xi^{-1}\theta_{0,S}+u_S\|_1-\|\xi^{-1}\theta_{0,S}\|_1\Big)\, .
\end{align*}
Notice that, for any $u\in \reals^p$, $\lim_{\xi\to 0}F(u;\xi) = \F*(u)$, where
\begin{align*}
\F*(u) \equiv \frac{1}{2}\<u,\hSigma u\> + \|u_{S^c}\|_1 + \<\sign(\theta_{0,S}),u_S\>\, .
\end{align*}
Indeed $F(u;\xi) = \F*(u)$ provided $\xi\le\min_{i\in S}|\theta_{0,i}/u_i|$. Further, $F(u;\xi)\ge \F*(u)$ for all $u$.

Let $\u*\equiv \arg\min_{u\in\reals^p} \F*(u)$, and set $\xistar \equiv \min_{i \in S} |\theta_{0,i}/u_{0,i}|$.
Then, for any $u\neq \u*$, and all
$\xi\in (0,\xistar)$, we have 
\begin{align*}
F(u;\xi)\ge \F*(u)> \F*(\u*) = F(\u*;\xi)\, .
\end{align*}
Hence $\u*$ is the unique minimizer of $F(u;\xi)$, i.e., $\hu(\xi) = \u*$ for all $\xi\in (0,\xistar)$.

It follows that $\thetaZN(\xi) = \theta_0+\xi\, \u*$ for all
$\xi\in (0,\xistar)$ and hence $\sign(\thetaZN(\xi)) = \v*$ and $\supp(\thetaZN(\xi))
=\T*$ where we set
\begin{eqnarray*}
v_{0,S} &\equiv &\sign(\theta_{0,S})\, ,\\
v_{0,S^c} &\equiv &\sign(u_{0,S^c})\, ,\\
\T* & \equiv & S\cup \supp(\u*)\, .
\end{eqnarray*}
Finally, the zero subgradient condition for $\u*$ reads $\hSigma \u* +
z = 0$, with $z_S = \sign(\theta_{0,S})$ and $z_{S^c}\in\partial\|
u_{0,S^c}\|_1$. In particular, $z_{\T*}=v_{0,\T*}$ and therefore $u_{0,\T*} =
-\hSigma_{\T*,\T*}^{-1}v_{\T*}$. This implies
\begin{align*}
\xistar \equiv \min_{i\in S}\left|\frac{\theta_{0,i}}{u_{0,i}}\right| =
\min_{i\in S} \left|\frac{\theta_{0,i}}{[\hSigma_{\T*,\T*}^{-1}v_{0,\T*}]_i}\right|\, .
\end{align*}

%%%%%%%%%%%%%%%%%%%%%%%%%%%%%%%%%%
\subsection{Proof of Lemma~\ref{lem:ZN-supp}}
\label{app:ZN-supp}
Writing the zero-subgradient conditions for problem~\eqref{eq:ZNLassoEstimator}, we have
\[
\hSigma (\thetaZN - \theta_0) = -\xi u, \quad \quad u \in \partial \|\thetaZN\|_1.
\]
Given that $T \supseteq S$, we have $\theta_{0,T^c} = 0$, and thus
 \begin{align*}
 \hSigma_{T,T} (\thetaZN_T - \theta_{0,T}) &= -\xi u_T\,, \\
  \hSigma_{T^c,T} (\thetaZN_T - \theta_{0,T}) &= -\xi u_{T^c}\,. 
 \end{align*}
Solving for $\thetaZN_T - \theta_{0,T}$
 in terms of $u_{T}$, we obtain
 \begin{align*}
 &\hSigma_{T^c,T} \hSigma_{T,T}^{-1} u_T = u_{T^c}\,,\\
 &\thetaZN_T = \theta_{0,T} -\xi \hSigma_{T,T}^{-1}u_T\,.
 \end{align*}
 This proves the `only if' part noting that $u_T = \sign(\thetaZN_T) = v_T$, and $\|u_{T^c}\|_\infty \le 1$ since $u \in \partial \|\thetaZN\|_1$.
 
 Now suppose that Eqs.~\eqref{eq:thetaZNvT1} and~\eqref{eq:thetaZNvT2} hold true.

Let $\tilde{\theta}_T = \theta_{0,T} - \xi \hSigma_{T,T}^{-1} v_T$, and $\tilde{\theta}_{T^c} = 0$.
We prove that $\tilde{\theta} = \thetaZN$, by showing that
it satisfies the zero-subgradient condition. By Eq.~\eqref{eq:thetaZNvT2}, $v_T = \sign(\tilde{\theta}_T)$.
Define $u \in \reals^p$  by letting $u_T = v_T$ and $u_{T^c} =  \hSigma_{T^c,T} \hSigma_{T,T}^{-1} v_T$. Note that $\|u_{T^c}\|_\infty \le 1$ by Eq.~\eqref{eq:thetaZNvT1},
and so $u \in \partial \|\tilde{\theta}\|_1$. Moreover,
\begin{align*}
\hSigma_{T,T} (\tilde{\theta}_T - \theta_{0,T}) &= -\xi u_T\,\\
\hSigma_{T^c,T} (\tilde{\theta}_T - \theta_{0,T}) &= -\xi u_{T^c}\,,
\end{align*}
Combining the above two equations, we get the zero-subgradient condition for $(\tilde{\theta}, u)$. Therefore, $\tilde{\theta} = \thetaZN$, and $v = \sign(\thetaZN)$.

%%%%%%%%%%%%%%%%%%%%%%%%%%%%%%%%%
\subsection{Proof of Lemma~\ref{lem:supp-lasso-DET}}
\label{app:supp-lasso-DET}
The proof proceeds along the same lines as the proof of Lemma~\ref{lem:ZN-supp}.
We begin with proving the `only if' part. The zero-subgradient condition for Problem~\ref{eq:LassoEstimator}
reads:
\[
-\frac{1}{n} \bX^\sT(Y - \bX \htheta^{n}) +\lambda u = 0\,, \quad \quad u \in \partial\|\htheta^{n}\|_1\,.
\]
Plugging for $Y = \bX \theta_0 + W$ and $\hr = (\bX^\sT W/n)$ in the above equation, we arrive at:
\[
\hSigma(\htheta^{n} - \theta_0) = \hr - \lambda u \,.
\]
Since $T \supseteq S$, $\theta_{0,T^c} = 0$, and writing the above equation for indices in $T$ and $T^c$ separately,
we obtain
\begin{align*}
\hSigma_{T^c,T}(\htheta^{n}_T - \theta_{0,T}) &= \hr_{T^c} - \lambda u_{T^c}\,,\\
\hSigma_{T,T} (\htheta^{n}_T - \theta_{0,T}) &= \hr_{T} - \lambda u_{T}\,.
\end{align*}
Solving for $\htheta^{n}_T - \theta_{0,T}$
from the second equation, we get
 \begin{align*}
\hSigma_{T^c,T} \hSigma_{T,T}^{-1} u_T &+ \frac{1}{\lambda} (\hr_{T^c} - \hSigma_{T^c,T} \hSigma_{T,T}^{-1} \hr_T) 
 = u_{T^c}\,,\\
 \htheta^{n}_T &= \theta_{0,T} - \hSigma_{T,T}^{-1}(\lambda u_T - \hr_T)\,.
 \end{align*}
 This proves Eqs.~\eqref{eq:z1-DET} and~\eqref{eq:z2-DET}, since $u_T = \sign(\htheta^{n}_T) = z_T$
 and $\|u_{T^c}\|_\infty \le 1$.
 
 We next prove the other direction. Suppose that Eqs.~\eqref{eq:z1-DET} and~\eqref{eq:z2-DET}
 hold true. Let $\tilde{\theta}_T = \theta_{0,T} - \hSigma_{T,T}^{-1} (\lambda z_T - \hr_T)$, and $\tilde{\theta}_{T^c} = 0$.
We prove that $\tilde{\theta} = \htheta^{n}$, by showing that
it satisfies the zero-subgradient condition. By Eq.~\eqref{eq:z2-DET}, $z_T = \sign(\tilde{\theta}_T)$.
Define $u \in \reals^p$  by letting $u_T = z_T$ and $u_{T^c} =  \hSigma_{T^c,T} \hSigma_{T,T}^{-1} z_T+ 
(\hr_{T^c} - \hSigma_{T^c,T} \hSigma^{-1}_{T,T} \hr_T) /\lambda$. Note that $\|u_{T^c}\|_\infty \le 1$ by Eq.~\eqref{eq:z2-DET},
and so $u \in \partial \|\tilde{\theta}\|_1$. Moreover,
\begin{align*}
\hSigma_{T,T} (\tilde{\theta}_T - \theta_{0,T}) &= -(\lambda u_T - \hr_T)\,\\
\hSigma_{T^c,T} (\tilde{\theta}_T - \theta_{0,T}) &= -(\lambda u_{T^c} - \hr_{T^c})\,,
\end{align*}
Combining the above two equations, we get the zero-subgradeint condition for $(\tilde{\theta}, u)$. Therefore, $\tilde{\theta} = \htheta^{n}$, and $z = \sign(\htheta^{n})$.

%%%%%%%%%%%%%%%%%%%%%%%%%%%%%%%%%%
\subsection{Proof of Lemma~\ref{lem:term2-DET}}
\label{proof:term2-DET}
Let $m = \hSigma_{\T*,\T*}^{-1} \hr_{\T*} = (\bX_{\T*}^\sT \bX_{\T*})^{-1} \bX_{\T*}^\sT W$. Conditioned on $\bX_{\T*}$, $m_i$ is a zero mean
Gaussian vector with variance $\sigma^2 \|e_i^\sT
(\bX_{\T*}\bX_{\T*})^{-1} \bX_{\T*}^\sT\|^2$. By a Gaussian  tail bound, we get
\[
\prob\Big( |m_i| \ge \sqrt{2c_1\log p}\, \sigma \|e_i^\sT (\bX_{\T*}^\sT \bX_{\T*})^{-1} \bX_{\T*}^\sT\| \Big) \le 2p^{-c_1}\,.
\]
Further, notice that $\|e_i^\sT (\bX_{\T*}^\sT \bX_{\T*})^{-1} \bX_{\T*}^\sT\|^2 \le \|\hSigma^{-1}_{\T*,\T*}\|_2/n$. By union bounding over $i=1,\dotsc,p$, we have
\[
\prob\Big( \|m\|_\infty \ge \sigma \sqrt{\frac{2c_1\log p}{n}}\, \|\hSigma^{-1}_{\T*,\T*}\|_2^{1/2} \Big) \le 2p^{1-c_1}\,.
\]
%  

%%%%%%%%%%%%%%%%%%%%%%%%%%%%%%%%%%
%\subsection{Proof of Proposition~\ref{pro:hrbound}}
%\label{proof:hrbound}
%
%Recalling the definition of $\hr$, we have $\hr_i = \<x_i,W\> /n$. Conditional on $\bX$, $\hr_i$
%is a zero mean Gaussian variable with standard deviation $\sigma \|\tx_i\|/n$. Using standard Gaussian
%tail bounds, we obtain the upper bound
%%
%\[
%\prob\bigg(|\hr_i| > \sqrt{2c_1\log p}\, \sigma \|x_i\|/ n\bigg) \le 2 p^{-c_1}\,.
%\] 
%%
%Moreover, $\|x_i\|^2 \deq \Sigma_{i,i} Z$, where $Z$ is a chi-squared random variable with $n$ degrees of freedom. Using the standard chi-squared
%tail bounds~\cite{Johnstone-chi2}, we have $\|x_i\|^2 < 2\Sigma_{i,i}\, n \le 2n$, with probability at least $1 - e^{-n/100}$. Consequently,
%%
%\[
%\prob\bigg(|\hr_i| > 2\sqrt{\frac{c_1 \log p}{n}} \,\sigma \bigg) \le  e^{-n/100} + 2 p^{-c_1}\,.
%\]
%%
%The result follows by applying union bound for $i= 1,\dotsc, p$.

%%%%%%%%%%%%%%%%%%%%%%%%%%%%%%%%%%
\subsection{Proof of Lemma~\ref{lem:mother-3}}
\label{proof:mother-3}

We begin by stating and proving  a lemma that is similar to
Lemma~5 in~\cite{Wainwright2009LASSO}, but provides a stronger control. 
\begin{lemma}\label{claim:svdI}
Let $\bZ\in\reals^{n\times k}$ be a random matrix with i.i.d. Gaussian
rows with zero mean and covariance $\Sigma$, with $k\ge 4$. Further let $a_1,\dots,
a_M\in\reals^k$ and $b_1,\dots,b_M\in\reals^k$ be non-random vectors. 
Then, letting $\hSigma_{\bZ}\equiv \bZ^{\sT}\bZ/n$, we have,  for all
$\Delta>0$:
\begin{align}
\prob\left\{\exists i\in [M]\;\; \mbox{{\rm s.t. }} \Big| \<a_i,
(\hSigma^{-1}_{\bZ} - \Sigma^{-1})b_i\>\Big|  \ge 8\sqrt{\frac{k}{n}}
|\<a_i,\Sigma^{-1}b_i\>|+\Delta\,\|\Sigma^{-1/2}a_i\|_2 
\|\Sigma^{-1/2}b_i\|_2 \right\}\nonumber\\ \le 2 e^{-\frac{k}{2}} + 2M
\,\exp\Big\{-\frac{n\Delta^2}{256}\Big\}\, .\label{eq:EntriesBound}
\end{align}
\end{lemma}
\begin{proof}
First notice that $\bZ= \widetilde{\bZ}\Sigma^{1/2}$ with
$\widetilde{\bZ}\in\reals^{n\times k}$  a random matrix with
i.i.d. standard Gaussian entries $Z_{ij}\sim\normal(0,1)$. By
substituting in the statement of the theorem, it is easy to check that
we only need to prove our claim in the case $\Sigma = \id_{k\times k}$
(i.e., for $\bZ$ with i.i.d. entries), which we shall assume hereafter.

Defining the event $\cE_* = \{\|\hSigma^{-1}-\id\|_2\le
8\sqrt{k/n}\}$, we have, by Eq.~(\ref{eq:diff-opt-spec2}) and the
union bound,
\begin{align*}
\prob&\left\{\exists i\in [M] \mbox{ s.t. } \Big| \<a_i,
(\hSigma^{-1} - \id)b_i\>\Big|  \ge  8\sqrt{\frac{k}{n}}
|\<a_i,b_i\>|+\Delta\,\|a_i\|_2 
\|b_i\|_2 \right\} \le\\ 
&2\,e^{-k/2} + M\, \max_{i\in [M]}
\prob\left\{ \big| \<a_i,
(\hSigma^{-1} - \id)b_i\>\big|  \ge 8\sqrt{\frac{k}{n}}
|\<a_i,b_i\>|+\Delta\,\|a_i\|_2 
\|b_i\|_2 ;\;\cE_*\right\} 
\end{align*}
We can now concentrate on the last probability. Let $\alpha \equiv
|\<a_i,b_i\>|$ and 
$\beta \equiv (\|a_i\|_2^2\|b_i\|_2^2-\<a_i,b_i\>^2)^{1/2}$. Since
$\hSigma$ is distributed as $R\hSigma R^{\sT}$ for any orthogonal
matrix $R$, we have
\begin{align*}
\<a_i,(\hSigma^{-1} - \id)b_i\> \ed \alpha\<e_1,(\hSigma^{-1} -
\id)e_1\> +\beta\<e_1,(\hSigma^{-1} -\id)e_2\>\, ,
\end{align*}
where $\ed$ denotes equality in distribution. Under the event $\cE_*$, we
have $|\alpha\<e_1,(\hSigma^{-1} -\id)e_1\>|\le 8\alpha\sqrt{k/n}$.
 Further $(\hSigma^{-1} -\id) = UDU^{\sT}$ with $U$
a uniformly random orthogonal matrix (with respect to Haar measure on the manifold of orthogonal matrices). Letting $u_1$, $u_2$ denote the
first two rows of $U$ we then have
\begin{align*}
\prob\left\{ \big| \<a_i,
(\hSigma^{-1} - \id)b_i\>\big|  \ge 8\sqrt{\frac{k}{n}}
|\<a_i,b_i\>|+\Delta\,\|a_i\|_2 
\|b_i\|_2 ;\;\cE_*\right\} \le \prob\{|\<u_1,Du_2\>|\ge\Delta;\;
\cE_*\}\, .
\end{align*}
Notice that conditioned on $u_2$ and $D$, $u_1$ is uniformly random on
a $(k-1)$-dimensional sphere. Further, letting $v_2=Du_2$, we have
$\|v_2\|_2\le 8\sqrt{k/n}$. Hence, by isoperimetric inequalities on
the sphere \cite{Ledoux}, we obtain
\begin{align*}
\prob\{|\<u_1,Du_2\>|\ge\Delta;\;
\cE_*\}&\le \sup_{\|v_2\|\le 8\sqrt{k/n}}\prob\{|\<u_1,v_2\>|\ge\Delta
|\, v_2\}\\
&\le 2\exp\Big\{-\frac{(k-2)\Delta^2}{128 k/n}\Big\}\le
2\exp\Big\{-\frac{n\Delta^2}{256}\Big\}\, ,
\end{align*}
where the last inequality holds for all $k\ge 4$. The proof is
completed by substituting this inequality in the expressions above.
\end{proof}

We are now in position to prove Lemma~\ref{lem:mother-3}.
\begin{proof}[Proof (Lemma \ref{lem:mother-3}).]
We apply Lemma \ref{claim:svdI} to $\hSigma=\hSigma_{\T*,\T*}$,
$M=t_0$, $a_i=e_i$ and $b_i=v_{0,\T*}$ for $i\in\{1,\dots,t_0\}$. We get
\begin{align*}
\prob\left\{ \exists i\in\T*\;\; \mbox{ s.t. }\big|[(\hSigma_{\T*,\T*}^{-1}
  - \Sigma_{\T*,\T*}^{-1}) v_{0,\T*}]_i\big| \ge
8\sqrt{\frac{t_0}{n}}\big|[\Sigma_{\T*,\T*}^{-1} v_{0,\T*}]_i\big| +
\Delta\|\Sigma_{\T*,\T*}^{-1/2}e_i\|_2\|\Sigma_{\T*,\T*}^{-1/2} v_{0,\T*}\|_2\right\}\le\\
2\, e^{-t_0/2}+2t_0\,\exp\Big\{-\frac{n\Delta^2}{256}\Big\}\, .
\end{align*}
Note that
$\|\Sigma_{\T*,\T*}^{-1/2}e_i\|_2\|\Sigma_{\T*,\T*}^{-1/2}v_{0,\T*}\|_2\le
C_{\min}^{-1} \|e_i\|_2\|v_{0,\T*}\|_2= C_{\min}^{-1}
\sqrt{t_0}$.
Further $|[\Sigma_{\T*,\T*}^{-1}v_{0,\T*}]_i\big| \ge 2c_2$, and hence
$\|\Sigma_{\T*,\T*}^{-1/2}e_i\|_2\|\Sigma_{\T*,\T*}^{-1/2}v
_{0,\T*}\|_2\le (1/2)\sqrt{c_*t_0}\,
|[\Sigma_{\T*,\T*}^{-1}v_{0,\T*}]_i\big|$.
We therefore get
\begin{align*}
\prob\left\{ \exists i\in\T*\mbox{ s.t. }\big|[(\hSigma_{\T*,\T*}^{-1}
  - \Sigma_{\T*,\T*}^{-1}) v_{0,\T*}]_i\big| \ge \Big(
8\sqrt{\frac{t_0}{n}} +\frac{\Delta}{2}\sqrt{c_*t_0}\Big)\big|[\Sigma_{\T*,\T*}^{-1} v_{0,\T*}]_i\big| \right\}\le\\
2\, e^{-t_0/2}+2t_0\,\exp\Big\{-\frac{n\Delta^2}{256}\Big\}\, .
\end{align*}
The proof is completed by taking $\Delta=16\sqrt{(c'\log p)/n}$.
\end{proof}

%%%%%%%%%%%%%%%%%%%%%%%%%%%%%%%%%%
\subsection{Proof of Lemma~\ref{lem:term4}}
\label{proof:term4}
By Lemma~\ref{lem:term2-DET}, we have
\[
\prob\Big( \|\hSigma_{\T*,\T*}^{-1} \hr_{\T*}\|_\infty \ge \sigma \sqrt{\frac{2c_1\log p}{n}}\, \|\hSigma^{-1}_{\T*,\T*}\|_2^{1/2} \Big) \le 2p^{1-c_1}\,.
\]
Recalling Eq.~\eqref{eq:event2-cons}, under the event $\event_1$ we have $\|\hSigma^{-1}_{\T*,\T*}\|_2 \le 9C_{\min}^{-1}$.
Since $\prob(\event_1^c) \le 2e^{-t_0/2}$, we arrive at:
\[
\prob\bigg( \|\hSigma_{\T*,\T*}^{-1} \hr_{\T*}\|_\infty \ge 3\sigma \sqrt{\frac{2c_1\log p}{n\, C_{\min}}} \bigg) \le  2p^{1-c_1} + 2e^{-\frac{t_0}{2}}\,. 
\]
%

%
%*********************************************************
%

\section{Generalized irrepresentability vs. irrepresentability}
\label{app:Example}

In this appendix we discuss the example provided in Section~\ref{subsec:example}
in more details. The objective is to develop some intuition on the domain of  validity
of generalized irrepresentability, and compare it with the standard
irrepresentability condition.
%
%********************************************************
%
%\subsection{A single confounding variable}

As explained in Section~\ref{subsec:example}, let $S =\supp(\theta_0) = \{1,\dots,s_0\}$ and consider
the following covariance matrix:
\begin{align*}
\Sigma_{ij} = \begin{cases}
1 & \mbox{if $i=j$},\\
a & \mbox{if $i=p, j\in S$ or $i\in S, j=p$,}\\
0 & \mbox{otherwise.}
\end{cases}
\end{align*}
Equivalently,
\begin{align*}
\Sigma = \id_{p\times p} + a\big(e_p u_S^{\sT} +u_Se^{\sT}_p\big)\, ,
\end{align*}
where $u_S$ is
the vector with entries $(u_S)_i=1$ for $i\in S$ and $(u_S)_i=0$ for
$i\not\in S$. It is easy to check that $\Sigma$ is strictly
positive definite for $a\in (-1/\sqrt{s_0},+1/\sqrt{s_0})$. By 
redefining the $p$-th covariate, we can assume, without loss of
generality, $a\in [0,+1/\sqrt{s_0})$.  
We will further assume $\sign(\theta_{0,i})=+1$ for all $i\in S$. 

This example captures the case of a single confounding variable,
i.e., of an irrelevant covariate that correlates strongly with the
relevant covariates, and with the response variable.

We will show that the Gauss-Lasso has a significantly broader domain
of validity with respect to the simple Lasso.
\begin{claim}
Consider the Gaussian design defined above, and suppose that
$a > 1/s_0$. Then for any regularization parameter $\lambda$ and 
for any sample size $n$, the probability of correct signed support 
recovery with Lasso is at most $1/2$. 
%The Lasso  estimate for the support of $\theta_0$ cannot be correct
%with high probability unless $a\in
%[0,1/s_0)$
(and is not guaranteed with high probability
unless $a\in [0,(1-\eta)/s_0]$, for some constant $\eta> 0$.
 
On the other hand, Theorem~\ref{thm:GLSuppRecovery} implies correct support
recovery with the Gauss-Lasso from $n=\Omega(s_0\log p)$ samples, for any
\begin{align}  
a\in
\left[0,\frac{1-\eta}{s_0}\right]\cup \left(\frac{1}{s_0},\frac{1-\eta}{\sqrt{s_0}}\right]\, . \label{eq:Example1Claim}
\end{align}
\end{claim}
\begin{proof}
In order to prove that Gauss-Lasso correctly recovers the support of $\theta_0$, we
will show that  all the conditions of Theorem
\ref{thm:LassoSuppRecovery} and Theorem~\ref{thm:GLSuppRecovery}
hold with constants of order one, provided
Eq.~(\ref{eq:Example1Claim}) holds.
Vice versa, the irrepresentability condition does not hold unless $a\in
[0,1/s_0)$, and hence the simple Lasso fails outside this regime.

We now proceed to check the assumptions of Theorems
\ref{thm:LassoSuppRecovery} and \ref{thm:GLSuppRecovery}, while showing that
irrepresentability does not hold for $a\ge 1/s_0$.

\vspace{0.2cm}

\noindent{\bf Restricted eigenvalues.} We have $\lambda_{\rm min}(\Sigma) = 1-a\sqrt{s_0}$.
In particular, for any set $T \subseteq[p]$, we have $\lambda_{\rm min}(\Sigma_{T,T}) \ge 1-a\sqrt{s_0} \ge \eta$.
Also, for any constant $c_0 \ge 0$, $\kappa(s_0,c_0) \ge 1-a\sqrt{s_0} \ge \eta$. 

\vspace{0.2cm}

\noindent{\bf Irrepresentability condition.} We have $\Sigma_{SS}
=\id_{s_0\times s_0}$ and hence
$\|\Sigma_{S^cS}\Sigma_{SS}^{-1}\|_{\infty} = \|\Sigma_{p,S}\|_1 =
as_0$. Hence the irrepresentability condition holds 
only if $a\in [0,1/s_0)$. The corresponding irrepresentability
parameter is $\eta =1-as_0$. 

For large $s_0$, the condition is only satisfied for a small interval
in $a$, compared to the interval for which $\Sigma$ is positive
definite.

\vspace{0.2cm}

\noindent{\bf Generalized irrepresentability condition.} In order to
check this condition, we need to compute $\T*$ and $\v*$ defined as
per Lemma \ref{lem:LimitXi0}. We have $\htheta^{\infty}(\xi) =
\arg\min_{\theta\in\reals^p} G(\theta;\xi)$
where 
\begin{align*}
G(\theta;\xi) &\equiv
\frac{1}{2}\,\<(\theta-\theta_0),\Sigma(\theta-\theta_0)\>+\xi\|\theta\|_1\\
&=\frac{1}{2}\|\theta-\theta_0\|_2^2 +
a\<u_S,(\theta_S-\theta_{0,S})\>\theta_p+\xi\|\theta\|_1\, .
\end{align*}
From this expression, it is immediate to see that
$\htheta^{\infty}_i(\xi) = 0$ for $i\not\in S\cup\{p\}$. Further
$\htheta^{\infty}_{S\cup\{p\}}(\xi)$ satisfies
\begin{align}
\theta_S-\theta_{0,S}+a\theta_p u_S +\xi v_S &= 0\, ,\label{eq:FirstExampleSol1}\\
\theta_p + a\<u_S,(\theta_S-\theta_{0,S})\> + \xi v_p &= 0\, ,\label{eq:FirstExampleSol2}
\end{align}
with $v_S\in\partial\|\theta_S\|_1$ and $v_p\in\partial |\theta_p|$.
Since $\theta_{0,S}>0$, we have, from Eq.~(\ref{eq:FirstExampleSol1}),
\begin{align*}
\htheta^{\infty}_S=\theta_{0,S} - (a\htheta^{\infty}_p+\xi) u_S\, ,
\end{align*}
provided $(a\htheta^{\infty}_p+\xi)\le \theta_{\rm min}$. Substituting
in Eq.~(\ref{eq:FirstExampleSol2}) and solving for $\theta_p$, we get
\begin{align*}
\htheta_p^{\infty}(\xi) = \begin{cases}
0 & \mbox{if $a\in [0,1/s_0)$}\\
\Big(\frac{as_0-1}{1-a^2s_0} \Big)\xi &\mbox{if $a\in [1/s_0,1/\sqrt{s_0})$.}
\end{cases}
\end{align*}
This holds provided $(a\htheta^{\infty}_p+\xi)\le \theta_{\rm min}$,
i.e., if $\xi\le \xi_{*}\equiv \min(1,(1-a^2s_0)/(1-a))\, \theta_{\min}$.

Using the definition in Lemma \ref{lem:LimitXi0}, we have
\begin{align*}
\T* = \begin{cases}
S & \mbox{if $a\in [0,1/s_0)$}\\
S\cup \{p\} &\mbox{if $a\in [1/s_0,1/\sqrt{s_0})$,}
\end{cases}
\end{align*}
and $v_{0,\T*} = u_{\T*}$.

We can now check the generalized irrepresentability condition. For 
$a\in [0,1/s_0)$ we have
$\|\Sigma_{\T*^c,\T*}\Sigma_{\T*,\T*}^{-1} v_{0,\T*}\|_{\infty}
=\|\Sigma_{S^c,S}\Sigma_{S,S}^{-1}u_{S}\|_{\infty}  = a s_0$, and
therefore the generalized irrepresentability condition is satisfied
with parameter $\eta = 1-as_0$. For 
$a\in [1/s_0,1/\sqrt{s_0})$, we have $\|\Sigma_{\T*^c,\T*}\Sigma_{\T*,\T*}^{-1}v_{0,\T*}\|_{\infty}
=0$. 

We therefore conclude that, for any fixed $\eta\in(0,1]$, the generalized
irrepresentability condition with parameter $\eta$ is satisfied for 
\begin{align*}
a\in \Big[0,\frac{1-\eta}{s_0}\Big]\cup
\Big[\frac{1}{s_0},\frac{1}{\sqrt{s_0}}\Big)\, ,
\end{align*}
a significant larger domain than for simple irrepresentability.

\vspace{0.2cm}

\noindent\emph{Minimum entry condition.} 
For $a\in [0,1/s_0)$, we have $\T*=S$ and it is therefore only
necessary to check Eq.~(\ref{eq:ConditionTheta1}). Since
$[\Sigma_{\T*,\T*}^{-1}v_{0,\T*}]_i=1$,  this reads
\begin{align*}
|\theta_{0,i}|\ge \Big(c_2+\frac{3}{2}\Big)\lambda =
C\sigma\sqrt{\frac{\log p}{n}}\, ,
\end{align*}
with $C$ a constant.

For $a\in (1/s_0, (1-\eta)/\sqrt{s_0}]$, we have $\T*=S\cup\{p\}$. 
A straightforward calculation shows that
\begin{align*}
\big|[\Sigma_{\T*,\T*}^{-1}v_{0,\T*}]_i\big| & =
\frac{1-a}{1-a^2s_0}\,,
\;\;\;\;\;\;
\mbox{for }i\in S\, ,\\
\big|[\Sigma_{\T*,\T*}^{-1}v_{0,\T*}]_{p}\big| & =
\frac{as_0-1}{1-a^2s_0} \,.\\
\end{align*}
%

%%
%\begin{align*}
%%
%\big|[\Sigma_{\T*,\T*}^{-1}v_{0,\T*}]_i\big| & =
%\frac{1}{2}\big(\lambda_1^{-1}+\lambda_2^{-1}\big)+\frac{1}{2\sqrt{s_0}}\big(\lambda_1^{-1}-\lambda_2^{-1}\big)\,
%\;\;\;\;\;\;
%\mbox{for }i\in S\, ,\\
%\big|[\Sigma_{\T*,\T*}^{-1}v_{0,\T*}]_{p}\big| & =
%\frac{1}{2}\big(\lambda_1^{-1}+\lambda_2^{-1}\big)+\frac{s_0-1}{2\sqrt{s_0}}\big(\lambda_1^{-1}-\lambda_2^{-1}\big)\, ,\\
%%
%\end{align*}
%%
%where
%%
%\begin{align*}
%%
%\lambda_1 = 1-a\sqrt{s_0}\, ,\;\;\;\;\;\;\;\;
%\lambda_2= 1+a\sqrt{s_0}\, .
%%
%\end{align*}
%%
It is not hard to show for all $a$ satisfying Eq.~(\ref{eq:Example1Claim}), we have 
\begin{align*}
\big|[\Sigma_{\T*,\T*}^{-1}v_{0,\T*}]_i\big| \le \frac{1}{1-(1-\eta)^2}\,  \;\;\;\mbox{ for }
i\in S,\;\;\;\;\;
\big|[\Sigma_{\T*,\T*}^{-1}v_{0,\T*}]_p\big| \ge C\, ,
\end{align*}
for some constant $C> 0$.
It therefore follows that condition (\ref{eq:ConditionTheta1}) holds
if $|\theta_{0,i}|\ge
C'\sigma\sqrt{\log p/n}$ and condition (\ref{eq:ConditionTheta2})
holds for $c_2 = C/2$.
\end{proof}

%%%%%%%%%%%%%%%%%%%%%%%%%%%%%%%%

%% that's all folks
\bibliographystyle{amsalpha}
\bibliography{all-bibliography}
\addcontentsline{toc}{section}{References}

\end{document}